\documentclass{amsproc}
\newtheorem{theorem}{Theorem}[section]
\newtheorem{lemma}[theorem]{Lemma}
\theoremstyle{definition}
\newtheorem{definition}[theorem]{Definition}
\newtheorem{example}[theorem]{Example}

\theoremstyle{remark}
\newtheorem{remark}[theorem]{Remark}
\newtheorem{corollary}[theorem]{Corollary}
\numberwithin{equation}{section}

\usepackage{graphicx, pdflscape}
\usepackage{amsmath}
 \usepackage{amssymb} %  for mathbb
 \usepackage{relsize, float, lineno}

 \usepackage{verbatim} % for 'comment'
 
 \usepackage{xcolor} % for colored comments
 \usepackage{color}
\usepackage{subcaption}
\usepackage{lipsum}
\allowdisplaybreaks

\begin{document}
\title[A generalized NSFD method for  autonomous dynamical systems]
{A generalized nonstandard finite difference method for a class of autonomous dynamical systems and its applications}
\author{Manh Tuan Hoang}
\address{Department of Mathematics, FPT University, Hoa Lac Hi-Tech Park, Km29 Thang Long Blvd, Hanoi, Viet Nam}
%School of Mathematics and Natural Sciences,
%University of Wuppertal, Germany}
\email{tuanhm14@fe.edu.vn; hmtuan01121990@gmail.com}
%    General info
%, 14E20; Secondary 46E25, 20C20}
%\date{January 1, 1994 and, in revised form, June 22, 1994.}

\dedicatory{This paper is dedicated in honor of Ronald E. Mickens' 80th birthday.}
%
%\linenumbers
\begin{abstract}
In this work, a class of continuous-time autonomous dynamical systems describing many important phenomena and processes arising in real-world applications is considered. We apply the nonstandard finite difference (NSFD) methodology proposed by Mickens to design a generalized NSFD method for the dynamical system models under consideration. This method is constructed based on a novel non-local approximation for the right-side functions of the dynamical systems.  It is proved by rigorous mathematical analyses that the NSFD method is dynamically consistent with respect to positivity, asymptotic stability and three classes of conservation laws, including direct conservation, generalized conservation and sub-conservation laws. Furthermore,  the NSFD method is easy to be implemented and can be applied to solve a broad range of mathematical models arising in real-life. Finally, a set of numerical experiments is performed to illustrate the theoretical findings and to show advantages of the proposed NSFD method.
\end{abstract}
\maketitle
%%%%%%%%%%%%%%%%%%%%%%%%%%%%%%%%%%%%%%%%%%%
%\section*{This is an unnumbered first-level section head}
%This is an example of an unnumbered first-level heading.
%\specialsection*{This is a Special Section Head}
%This is an example of a special section head%
%%%%%%%%%%%%%%%%%%%%%%%%%%%%%%%%%%%%%%%%%%%%%%%%%%%%%%%%%%%%%%%%%%
\section{Introduction}\label{intro}
Many important problems in real-world situations can be mathematically modeled by autonomous dynamical systems of the general form:
%%%%
%%%
\begin{equation}\label{eq:1}
\dfrac{dy(t)}{dt} = f(y(t)), \quad t \geq 0, \quad \quad y(0) = y_0 \in \mathbb{R}^n,
\end{equation}
where $y$ is a vector function of $t$ with $n$ components; $f$ is a function of $y$ and it is assumed to satisfy suitable conditions that guarantee solutions of the model \eqref{eq:1} exist and are unique (see, for instance, \cite{Allen, Brauer, Edelstein-Keshet, Khalil, Martcheva, Mattheij, Smith, Stuart}). The solutions of the model \eqref{eq:1} often possess  several important properties, typically positivity, asymptotic stability, boundedness, periodicity and physical properties. \par
In this work, we consider the model \eqref{eq:1} under the following two conditions:
\begin{itemize}
\item \textbf{(C1):} There exists a real number $\alpha$ such that
\begin{equation}\label{eq:2}
f(y) + \alpha y \geq 0 \quad \mbox{for all} \quad y \geq 0.
\end{equation}
Here the notation $\geq$  is meant entry-wise for vectors.
\item \textbf{(C2):} The set of equilibria is finite and all the equilibria are hyperbolic.
\end{itemize}
The condition \textbf{(C1)} implies that the set $\mathbb{R}^n_+ = \big\{y \in \mathbb{R}^n| y \geq 0\big\}$ is a positively invariant set of the model \eqref{eq:1}, that is, $y(t) \geq 0$ whenever $y_0 \geq 0$ (see \cite{Horvath, Smith}), meanwhile, the condition \textbf{(C2)} means that the local asymptotic stability (LAS) of all the equilibria can be determined by the Lyapunov's indirect method (see \cite{Khalil, LaSalle, Lyapunov, Stuart}).  Also, in many cases the model \eqref{eq:1} can satisfy three classes of conservation laws, which include direct conservation, generalized conservation and sub-conservation laws \cite{Mickens6}. {It is easy to find a large number of mathematical models in real-world applications with the above-mentioned properties. Below, we list some notable models that satisfy the conditions \textbf{(C1)} and \textbf{(C2)} and/or the conservation laws
\begin{itemize}
\item mathematical models in biology, ecology and epidemiology \cite{Allen, Brauer, Edelstein-Keshet, Martcheva};
\item  models for interacting species \cite{Brauer};
\item chemostat systems \cite{Smith};
\item epidemic models for infectious diseases \cite{Hethcote, Mena-Lorca};
\item  metapopulation models \cite{Amarasekare, Keymer};
\item single-species population models \cite{Cooke};
\item predator-prey models with linear prey growth and Beddington-DeAngelis functional response \cite{DeAngelis, Dimitrov4};
\item a simple vaccination model with multiple endemic states \cite{Kribs-Zaleta};
\item a vaccination model with non-linear incidence \cite{Gumel1};
\item a two-stage epidemic model with generalized non-linear incidence \cite{Moghadas};
\item a generalization of the Kermack-McKendrick epidemic model \cite{Capasso}.
\end{itemize}
{It should be emphasized that} dynamically consistent NSFD schemes for most of the above models have not been studied.\par
}
%%%
%%
Our main objective is to construct a numerical method preserving the positivity, LAS  and conservation laws of the model \eqref{eq:1} for all the values of the step size. To achieve this goal,  we will apply the \textit{Mickens' methodology} \cite{Mickens1, Mickens2, Mickens3, Mickens4, Mickens5} to propose a dynamically consistent \textit{nonstandard finite difference (NSFD)} method for the model \eqref{eq:1}. We recall that the concept of NSFD schemes was first introduced by Mickens to overcome usual numerical instabilities caused by standard finite difference (SFD) schemes as well as drawbacks of SFD schemes \cite{Mickens1, Mickens2, Mickens3, Mickens4, Mickens5, Mickens0}. One of the main and outstanding advantages of NSFD schemes is that they have the ability to preserve essential mathematical properties of the corresponding differential equation models regardless of the values of the step size; furthermore, they are easy to be implemented. Therefore, NSFD schemes are suitable and effective to simulate dynamics of differential equation models over very long time periods. During the past several decades, NSFD schemes have been strongly developed by mathematicians and engineers and have become one of the most powerful methods for solving differential equations. Nowadays, NSFD schemes have been widely used for ordinary differential equations, partial differential equations, delay differential equations, fractional differential equations and integro-differential equations (see, for instance, \cite{Adamu, Adekanye, Anguelov, Anguelov1, Anguelov2, Anguelov3, Arenas, Chapwanya, Cresson, Cresson1, Dimitrov1, Dimitrov2, Dimitrov3, Ehrhardt, Fatoorehchi, Garba, Hoang1, Kojouharov, Mickens1, Mickens2, Mickens3, Mickens4, Mickens5, Mickens0, Mickens6, Mickens7, Mickens8, Mickens9, Mickens10, Mickens11, Mickens12, Mickens12, Mickens13, Mickens14, Mickens15, Mickens16, Patidar1, Patidar2, Roeger1, Roeger2, Roeger3, Roeger4, Verma, Wood0, Wood1, Wood2} and references therein). Here, we refer the readers to \cite{Mickens4, Patidar1, Patidar2} for notable reviews of NSFD schemes.\par
It is well-known that  the positivity and LAS are essential properties of most systems in biology, ecology and epidemiology (see, for example, \cite{Allen, Brauer, Edelstein-Keshet, Khalil, Martcheva, Smith, Stuart}). {For this reason, the construction of NSFD schemes} preserving these properties is very important and has attracted the attention of many researchers. Below,  we mention some notable results in this topic.\\
In an early work \cite{Anguelov}, Anguelov and Lubuma introduced an NSFD method preserving monotonic properties of solutions of differential equations and the LAS of hyperbolic equilibrium points. In 2005 and 2007, Dimitrov and Kojouharov in \cite{Dimitrov1, Dimitrov2} analyzed stability-preserving NSFD schemes based on the $\theta$-methods and the second-order Runge-Kutta methods for general autonomous dynamical systems. Note that the NSFD schemes constructed in \cite{Anguelov, Dimitrov1, Dimitrov2} are not positive. In  \cite{Dimitrov3, Wood0}, dynamically consistent NSFD schemes for general productive-destructive systems were formulated. In 2015, Wood and Kojouharov in \cite{Wood1} designed an NSFD method preserving the positivity of solutions and the local behavior of dynamical systems near equilibria. In \cite{Cresson}, Cresson and Pierret studied the construction of NSFD schemes for a general class of two-dimensional dynamical systems including many models in population dynamics. In 2020, Dang and Hoang constructed and analyzed high-order nonstandard Runge-Kutta methods for a class of autonomous dynamical systems possessing the positivity and LAS. Very recently, Hoang in \cite{Hoang} proposed a second-order NSFD method that simultaneously preserves  the positivity and LAS of one-dimensional autonomous dynamical systems. Besides, many results focusing on NSFD schemes preserving the positivity and LAS of specific models in biology, ecology, biology and other fields are also remarkable (see, for example, \cite{Adamu, Adekanye, Anguelov2, Roeger1, Roeger2, Wood2} and references therein).\par
%%%
%%
%%
Motivated and inspired by the above reasons, we will apply the NSFD methodology proposed by Mickens to construct a dynamically consistent NSFD scheme for the model \eqref{eq:1}. This scheme is constructed based on a novel non-local approximation for the right-side functions.  It is proved by rigorous mathematical analyses that the NSFD scheme preserves the positivity, asymptotic stability and conservation laws for all the values of the step size.  Moreover, the constructed NSFD scheme is easy to be implemented and can be used to solve a broad range of mathematical problems arising in real-life. This claim is also confirmed by a set of illustrative numerical examples.\par
The plan of this work is as follows:\\
Section \ref{sec2} provides some basic concepts and preliminaries. The NSFD scheme is introduced and analyzed in Section \ref{sec4}. A set of numerical experiments is investigated and reported in Section \ref{sec5}. Finally, conclusions, remarks and some open problems are given in Section \ref{sec6}.
%%%
\section{Definitions and Preliminaries}\label{sec2}
%%\cite{Allen, Brauer, Edelstein-Keshet,Elaydi,  Khalil, LaSalle, Lyapunov,  Martcheva, Mattheij, Mickens1, Mickens2, Mickens3, Mickens4, Mickens5, Mickens0, Mickens6, Smith, Stuart}
In this section, we recall from literature some essential concepts and preliminaries related to qualitative theory of continuous and discrete dynamical systems, conservation laws and NSFD schemes.
\subsection{Stability of equilibria of continuous-time dynamical systems}
Consider the continuous-time dynamical system \eqref{eq:1}. We recall that a point $y^* \in \mathbb{R}^n$ is called {an equilibrium point of System \eqref{eq:1}} if $f(y^*) = 0$ (see \cite{Khalil, Stuart}). Without loss of generality, we can always assume that $y^* = 0$.
\begin{definition}[\cite{Khalil, Stuart}]\label{def2.1}
The equilibrium point $y^* = 0$ of \eqref{eq:1} is said to be:
\begin{itemize}
\item stable, if, for each $\epsilon > 0$, there is $\delta = \delta(\epsilon)$ such that
\begin{equation*}
\|y(0)\| < \delta \quad \mbox{implies that} \quad \|y(t)\| < \epsilon, \quad \forall t \geq 0;
\end{equation*}
\item unstable if it is not stable;
\item (locally) asymptotically stable if it is stable and $\delta$ can be chosen such that
\begin{equation*}
\|y(0)\| < \delta \quad \mbox{implies that} \quad \lim_{t \to \infty}y(t) = 0.
\end{equation*}
\end{itemize}
\end{definition}
The following theorem is known as the Lyapunov's indirect method \cite{Khalil, LaSalle, Lyapunov, Stuart}.
\begin{theorem}\label{theorem1}
Let $y^* = 0$ be an equilibrium point for the system \eqref{eq:1}, where $f: D \to \mathbb{R}^n$ is continuously differentiable and $D$ is a neighborhood of the origin. Let
\begin{equation*}
A = \dfrac{\partial f}{\partial y}(y)\bigg|_{y = 0}.
\end{equation*}
\end{theorem}
Then,
\begin{itemize}
\item $y^*$ is LAS if $Re \lambda_i < 0$ for all eigenvalues of $A$;
\item $y^*$ is unstable if $Re \lambda_i > 0$ for one or more the eigenvalues of $A$.
\end{itemize}
{
\begin{remark}\label{remark2.2}
%\begin{itemize}
An equilibrium point $y^*$ of \eqref{eq:1} is said to be hyperbolic if none of the eigenvalues of the matrix $A$ lies on the imaginary axis and non-hyperbolic otherwise (see \cite{Stuart}). Hence, Theorem \ref{theorem1} is only applicable for hyperbolic equilibrium points.
%\end{itemize}
\end{remark}
}
%%%%%%%%%%%%%%
\begin{theorem}[\cite{Horvath, Smith}]\label{theorem2.1}
Suppose that $f$ in \eqref{eq:2}  has the property that solutions of initial value problems $y(0) = y_0 \geq 0$ are unique and, for all $i$, $f_i(y) \geq 0$ whenever $y \geq 0$ satisfies $y_i = 0$. Then $y(t) \geq 0$ for all $t \geq 0$ for which it is defined, provided $y(0) \geq 0$.
\end{theorem}
\begin{remark}
We infer from Theorem \ref{theorem2.1}  that the set $\mathbb{R}^n_+$ is a positively invariant set of the model \eqref{eq:1} if the condition \eqref{eq:2} holds.
\end{remark}
%%%%%%%%%%%%%%%%
\subsection{Conservation laws}
We consider three classes of conservation laws for the model \eqref{eq:1}, that are  direct conservation, generalized conservation and sub-conservation laws \cite{Mickens6}.\par
{For the system \eqref{eq:1}, let us denote by $P_T(t) = y_1(t) + y_2(t) + \ldots + y_n(t)$.} We say that the system \eqref{eq:1} satisfies
\begin{itemize}
\item a direct conservation law (DCL) if
\begin{equation*}
\dfrac{dP_T}{dt} = 0,
\end{equation*}
which implies that $P_T$ is constant;
\item a generalized conservation law  (GCL) if
\begin{equation*}
\dfrac{dP_T}{dt} = a_1 - b_1P_T, \quad a_1, b_1 > 0,
\end{equation*}
which implies that $P_T(t)$ monotonically converges to $a_1/b_1$ as $t \to \infty$;
\item a sub-conservation law (SCL) if a sub-set of the populations, $P_S(t)$, where
\begin{equation*}
P_S(t) = P_{1S}(t) + P_{2S}(t)  + \ldots + P_{mS}(t), \quad 2 \leq m \leq n,
\end{equation*}
has the property
\begin{equation*}
\dfrac{dP_T}{dt} = 0,
\end{equation*}
or
\begin{equation*}
\dfrac{dP_T}{dt} = a_1 - b_1P_T.
\end{equation*}
\end{itemize}
In \cite{Mickens6}, Mickens and  Washington introduced NSFD discretizations of {some mathematical models satisfying the above conservation laws.}
\subsection{Stability of equilibria of discrete-time dynamical systems}
Consider a general discrete-time dynamical system governed by first-order difference equations of the form
\begin{equation}\label{eq:3}
y_{k + 1} = g(y_k), \quad y_0 \in \mathbb{R}^n,
\end{equation}
where $g: D \to \mathbb{R}$ and $D \subset \mathbf{R}^n$ is the domain of the function $g$. The vector $y_0$ is called the initial value or initial data. Note that the uniqueness of the sequence $\{y_k\}_{k = 0}^{\infty}$ is automatic if it exists.\par
A point $y^* \in \mathbf{R}^n$ is called a fixed point or an equilibrium point of the system \eqref{eq:3} if $g(y^*) = y^*$ (see \cite{Elaydi, Stuart}). 
\begin{definition}[\cite{Elaydi, Stuart}]\label{def2.2}
Let $y^*$ be an equilibrium point of the system \eqref{eq:3}. Then, $y^*$ is said to be:
\begin{itemize}
\item stable if given $\epsilon > 0$ there exists $\delta > 0$ such that
\begin{equation*}
\|y_ - y^*\| < \delta \quad \mbox{implies} \quad \|y_k - y^*\| < \epsilon,
\end{equation*}
for all $k > 0$.
\item unstable if it is not stable;
\item attracting if there exists $\eta > 0$ such that
\begin{equation*}
\|y_0 - y^*\| < \eta \quad \mbox{implies} \quad \lim_{k \to \infty}y_k = y^*;
\end{equation*}
\item (locally) asymptotically stable if it is stable and attracting.
\end{itemize}
\end{definition}
\begin{theorem}[\cite{Elaydi, Stuart}]\label{theorem2.2}
Let $f \in \mathcal{C}^2(\mathbb{R}^n, \mathbb{R}^n)$. Then a fixed point $y^*$ of the system \eqref{eq:3} is locally asymptotically stable if the eigenvalues of $df(y^*)$ lie strictly inside the unit circle. If any of the eigenvalues lie outside the unit circle the fixed point is unstable.
\end{theorem}
\begin{remark}
A fixed point $y^*$ is said to be hyperbolic if none of the eigenvalues of the matrix $df(y^*)$ lies on the unit circle and non-hyperbolic if otherwise (see \cite{Stuart}). Hence,  Theorem \ref{theorem2.2} only applicable for hyperbolic equilibrium points.	
\end{remark}
\subsection{Nonstandard finite difference schemes}
Consider a finite difference scheme  that approximates solutions of \eqref{eq:1} in the form
\begin{equation}\label{eq:4}
D_{\Delta t}(y_k) = F_{\Delta t}(f; y_k),
\end{equation}
where $D_{\Delta t}(y_k) \approx dy/dt$, $F_{\Delta t}(f; y_k) \approx f(y)$ and $t_k = k\Delta t$, $\Delta t$ is the step size.\par
According to the Mickens' methodology \cite{Mickens1, Mickens2, Mickens3, Mickens4, Mickens5, Mickens0}, an NSFD scheme for the equation \eqref{eq:1} is a discrete model constructed based on a set of six rules. In particular, NSFD schemes for first-order differential equations can be defined as follows \cite{Anguelov, Anguelov1, Dimitrov2}.
\begin{definition}
The finite difference scheme \eqref{eq:4} is called an NSFD scheme if at least one of the following conditions is satisfied:
\begin{itemize}
\item $D_{\Delta t}(y_k) = \dfrac{y_{k + 1} - y_k}{\phi(\Delta t)}$, where $\phi(\Delta t) = \Delta t + \mathcal{O}(\Delta t^2)$ is a non-negative function and is called a nonstandard denominator function;
\item $F_{\Delta t}(f; y_k) = g(y_k, y_{k + 1}, \Delta t)$, where $g(y_k, y_{k + 1}, \Delta t)$ is a non-local approximation of the right-hand side of the system \eqref{eq:1}.
\end{itemize}
\end{definition}
\begin{definition}[\cite{Anguelov1, Anguelov2}]
Assume that the solutions of the equation \eqref{eq:1} satisfy some property $\mathcal{P}$. The numerical scheme \eqref{eq:4} is called (qualitatively) stable with respect to property $\mathcal{P}$ (or $\mathcal{P}$-stable), if for every value of $\Delta t > 0$ the set of solutions of \eqref{eq:4} satisfies property $\mathcal{P}$.
\end{definition}
%%%
%
%%
\begin{definition}[\cite{Anguelov, Anguelov1, Mickens0}]
Consider the differential equation $dy/dt = f(y)$. Let a finite difference scheme for the equation be $y_{k+1} = F(y_k; \Delta t)$. Let the differential equation and/or its solutions have property $\mathcal{P}$. The discrete model equation is dynamically consistent with the differential equation if it and/or its solutions also have property $\mathcal{P}$.
\end{definition}
{
\begin{definition}(see \cite{Wood1})
The NSFD method \eqref{eq:4} is called \textit{positive }or \textit{positivity-preserving}, if, for any value of the step size $\Delta t$, and $y_0 \in \mathbb{R}_+^n$ its solution remains positive, i.e., $y_k \in \mathbb{R}_+^n$ for all $k \in \mathbb{N}$.
\end{definition}
}
\section{Construction of the NSFD scheme}\label{sec4}
In this section, we construct a generalized NSFD scheme for the model \eqref{eq:1} and investigate its qualitative dynamical properties.
\subsection{Derivation of the NSFD scheme and its positivity}\label{sub4.1}
%%%%%%%%%%%%%%%%%%%%%%%%%%%%%%%%%%%%%%%
First, suppose that the function $f$ satisfies the condition \textbf{(C1)} with a real number $\alpha$. We use the following non-local approximation for the function $f$
\begin{equation}\label{eq:11}
f(y(t_k)) = f(y(t_k)) + \big(my(t_k) -my(t_k)\big) \approx f(y_k) + my_k - my_{k + 1},
\end{equation}
where $m$ is a real number. By combining \eqref{eq:11} with the nonstandard approximation for the first-order derivative $dy/dt$, we obtain the following NSFD model for the system \eqref{eq:1}:
\begin{equation}\label{eq:16}
\dfrac{y_{k + 1} - y_k}{\phi(\Delta t)} = f(y_k) + my_k - my_{k + 1},
\end{equation}
%%%
%%%
where $\phi(\Delta t) = \Delta t + \mathcal{O}(\Delta t^2)$ as $\Delta t \to 0$. The explicit form of the scheme \eqref{eq:16} is given by
\begin{equation}\label{eq:17}
y_{k + 1} = \dfrac{y_k + \phi f(y_k) + \phi my_k}{1 + m\phi},
\end{equation}
or equivalently,
\begin{equation}\label{eq:18}
y_{k + 1} = y_k + \dfrac{\phi}{1 + m\phi}f(y_k).
\end{equation}
{
\begin{remark}
Although the rule of nonlocal approximation of the NSFD methodology states that nonlinear terms in the right-hand side $f(y)$ should be approximated in a nonlocal way, in many cases linear terms with negative coefficients, namely, $-cy$ ($c > 0$) can be also nonlocally approximated by $-cy_{k+1}$ to guarantee the positivity of the corresponding NSFD schemes (see, for instance, \cite{Adamu, Anguelov2, Arenas, DangHoang1, Hoang3, Hoang5, Martin-Vaquero1, Martin-Vaquero2, Mickens6}). Hence, the approximation \eqref{eq:11} can be considered as a nonlocal approximation, which follows the NSFD methodology.
\end{remark}
}
%%%%%%%%%%%%%%%%%%%%%%%%%%%%%%%%%%%%%%%%%%%%%%%%%%%%%%%%%%%%%%%%%%%%%%%%%%%
%%%%%%%%%%%%%%%%%%%%%%%%%%%%%%%%%%%%%%%%%%%%%%%%%%%%%%%%%%%%%%%%%%%%%
%%%
\begin{lemma}\label{lemma4.1}
The NSFD sheme \eqref{eq:16} preserves the positivity of the model \eqref{eq:1} for all the values of the step size $\Delta t$ whenever
\begin{equation}\label{eq:12}
m \geq m_P := \max\{\alpha, \,0\}.
\end{equation}
\end{lemma}
%%%%%%%%%%%%%
\begin{proof}
We use mathematical induction to prove the positivity of the NSFD scheme \eqref{eq:16}. Indeed, assume that $y_k \geq 0$. Then,
\begin{equation*}
f(y_k) + my_k \geq f(y_k) + \alpha y_k \geq 0.
\end{equation*}
Combining this estimate with \eqref{eq:17}, we obtain
\begin{equation*}
y_{k + 1} = \dfrac{y_k + \phi(f(y_k) + my_k)}{1 + m\phi} \geq 0.
\end{equation*}
This is the desired conclusion. The proof is complete.
\end{proof}
%%
%%%
%%
\begin{remark}
The discretization \eqref{eq:11} guarantees that the NSFD scheme \eqref{eq:16} is consistent. Similarly to convergence analysis presented in \cite{Cresson}, we can conclude that the NSFD scheme \eqref{eq:16} is convergent of order $1$.
\end{remark}
{
\begin{remark}\label{Remark3.4new}
It is easy to see from \eqref{eq:18} that the NSFD scheme \eqref{eq:16} can be rewritten in the form
\begin{equation*}
\dfrac{y_{k + 1} - y_k}{\Phi(\Delta t)} = f(y_k),
\end{equation*}
where
\begin{equation*}
\Phi(\Delta t) = \dfrac{\phi(\Delta t)}{1 + m\phi(\Delta t)} = \phi(\Delta t) + \mathcal{O}(\phi^2(\Delta t)) = \Delta t + \mathcal{O}(\Delta t^2).
\end{equation*}
Hence, the NSFD scheme \eqref{eq:16} can be equivalently derived by using renormalization of the denominator. Here, it is important to note that conditions for the NSFD scheme \eqref{eq:16} to be dynamically consistent with the model \eqref{eq:1} will be set to $m$. 
\end{remark}
}
\subsection{Equilibria and their asymptotic stability}\label{sub4.2}
Let us denote by $\mathcal{E}_C$ and $\mathcal{E}_D$ the sets of equilibrium points of the continuous-time dynamical system \eqref{eq:1} and the discrete-time model \eqref{eq:16}, respectively. It should be emphasized that $\mathcal{E}_C \subset \mathcal{E}_D$ for most standard finite difference schemes such as the Taylor's and Runge-Kutta schemes since they often generate spurious equilibrium points, which depend on the step size \cite{Mickens1, Mickens2, Mickens3, Mickens4, Mickens5, Mickens0}. However, we infer from the system \eqref{eq:18} that:
\begin{lemma}\label{lemma4.3}
$\mathcal{E}_C = \mathcal{E}_D$ for all the  finite values of the step size $\Delta t$. In other words, the NSFD scheme \eqref{eq:16} preserves the set of equilibrium points of the model \eqref{eq:1}.
\end{lemma}
%%%%%%%%%%%%%%%%%%%%%
%%%%%%%%
%
Let us denote by $\mathcal{E}_C^S$ the set of all asymptotically stable equilibrium points of the system \eqref{eq:1}. To analyze the stability of the NSFD scheme \eqref{eq:16}, we introduce the following notation
\begin{equation}\label{eq:16a}
\begin{split}
&\mathcal{F} = \bigcup_{y^* \in \mathcal{E}_C^S}\sigma(J_C(y^*)),\\
&m_S = \max_{\lambda \in \mathcal{F}}\bigg\{-\dfrac{|\lambda|^2}{2Re\lambda}\bigg\},
\end{split}
\end{equation}
where $J_C(y^*)$ is the Jacobian matrix of the system \eqref{eq:1} evaluating at $y^*$ and $\sigma(J_C(y^*))$ is the set of eigenvalues of $J_C(y^*)$. Note that $Re\lambda < 0$ for all $\lambda \in \mathcal{F}$ (see Theorem \ref{theorem1}).
%%%
\begin{theorem}\label{theorem4.1}
Let $m$ be a real number with the property that
\begin{equation}\label{eq:16c}
m \geq m_S,
\end{equation}
where $m_S$ is given in \eqref{eq:16a}. Then, the NSFD scheme \eqref{eq:16} preserves the LAS of the system \eqref{eq:1} for all the values of the step size $\Delta t$.
\end{theorem}
\begin{proof}
Let $y^*$ be any equilibrium point of the continuous system \eqref{eq:1}. Let us denote by ${J}_C(y^*)$ and ${J}_D(y^*)$ the Jacobian matrices of the systems \eqref{eq:1} and \eqref{eq:16} evaluating at $y^*$, respectively. Then, it follows from the equation \eqref{eq:18} that
\begin{equation}\label{eq:20}
{J}_D = I +  \dfrac{\phi}{1 + m\phi}{J}_C,
\end{equation}
where $I$ is the identity matrix. The relation \eqref{eq:20} implies that
\begin{equation*}\label{eq:21}
\begin{split}
\det(xI - {J}_C) &= \det\bigg[\bigg(x + \dfrac{1 + m\phi}{\phi}\bigg)I - \frac{1 + m\phi}{\phi}J_D\bigg]\\
 &= \bigg(\dfrac{1 + m\phi}{\phi}\bigg)^n\det\bigg[\bigg(1 + \dfrac{\phi x}{1 + m\phi}\bigg)I - {J}_D\bigg].
\end{split}
\end{equation*}
Hence, $\lambda$ is an eigenvalue of ${J}_C$ if and only if $\mu = 1 + \dfrac{\phi}{1 + m\phi}\lambda$ is an eigenvalue of ${J}_D$. This implies that
\begin{equation}\label{eq:22}
\begin{split}
|\mu_i|^2 &= \bigg|1 + \dfrac{\phi}{1 + m\phi}\lambda_i\bigg|^2 = \bigg(1 + \dfrac{\phi}{1 + m\phi}Re\lambda_i\bigg)^2 + \bigg(\dfrac{\phi}{1 + m\phi}Im\lambda_i\bigg)^2\\
&= 1 + \dfrac{2\phi}{1 + m\phi}Re\lambda_i + \bigg(\dfrac{\phi}{1 + m\phi}\bigg)^2|\lambda_i|^2 .
\end{split}
\end{equation}
We now consider the following two cases of the stability of $y^*$.\\
\textbf{Case 1.} $y^*$ is an unstable equilibrium point of the system \eqref{eq:1}. In this case, it follows from Theorem \ref{theorem1} that there exists $\lambda_i \in \sigma({J}_C(y^*))$ such that $Re \lambda_i > 0$. Then, we infer from \eqref{eq:22} that the corresponding eigenvalue $\mu_i$ of ${J}_D(y^*)$ satisfies
{
\begin{equation*}
|\mu_i|^2 =  1 + \dfrac{2\phi}{1 + m\phi}Re\lambda_i + \bigg(\dfrac{\phi}{1 + m\phi}\bigg)^2|\lambda_i|^2 > 1,
\end{equation*}
}
which implies that $y^*$ is an unstable equilibrium point of the system \eqref{eq:16} (see Theorem \ref{theorem2.2}).\\
\textbf{Case 2.} $y^*$ is an asymptotically stable equilibrium point of the system \eqref{eq:1}. In this case, $Re\lambda < 0$ for all $\lambda \in \sigma({J}_C(y^*))$. By using \eqref{eq:22}, we obtain $|\mu|^2 < 1$ if and only if
%%%
\begin{equation*}
\big(2mRe\lambda + |\lambda|^2\big)\phi < -2Re\lambda.
\end{equation*}
This condition is always satisfied if $m \geq -{|\lambda|^2}/{(2Re\lambda)}$. Hence, we obtain from the condition $m \geq m_s$ and Theorem \ref{theorem2.2} that $y^*$ is an asymptotically stable equilibrium point of the NSFD scheme \eqref{eq:16}.\par
Finally, the proof is completed by combining Cases $1$ and $2$.
\end{proof}
%%%
{
\begin{remark}
In many cases, it is not easy to obtain conditions for NSFD schemes to be dynamically consistent with respect to the asymptotic stability of dynamical systems (see, for instance, \cite{Cresson, DangHoang1, DangHoang3, DangHoang5, Dimitrov3, Hoang3, Hoang5, Wood0}). However, we can obtain the stability threshold $m_S$ of the NSFD scheme \eqref{eq:16} by simple calculations or a simple numerical algorithm.
\end{remark}
}
\subsection{Conservation laws}\label{sub4.3}
Assume that the system \eqref{eq:1} satisfies conservation laws. We show that the NSFD scheme \eqref{eq:16} has the ability to preserve this property.
\begin{lemma}\label{lemma4.6}
The NSFD scheme \eqref{eq:16} preserves the DCL for all finite step sizes.
\end{lemma}
\begin{proof}
Assume that the system \eqref{eq:1} satisfies a DCL, that is the total population satisfies
\begin{equation*}
\dfrac{dP_T}{dt} = \dfrac{dy_1}{dt} + \dfrac{dy_2}{dt} + \ldots + \dfrac{dy_n}{dt} = 0,
\end{equation*}
which implies that $\sum_{i = 1}^ny_i(t)$ is constant, and hence, $\sum_{i = 1}^nf_i(y(t)) = 0$ for all $t \geq 0$. By using \eqref{eq:18} we obtain
\begin{equation*}
P_{T, k + 1} := \sum_{i = 1}^ny_{i, k+1} = \sum_{i = 1}^ny_{i, k} + \dfrac{\phi}{1 + m\phi}\sum_{i = 1}^nf_i(y_k) =  \sum_{i = 1}^ny_{i, k+1} := P_{T,k}.
\end{equation*}
This implies that the DCL is preserved automatically by the NSFD scheme \eqref{eq:16}. The proof is complete.
\end{proof}
Next, assume that a GCL exists for the system \eqref{eq:1}, that is
\begin{equation}\label{eq:24}
\dfrac{dP_T}{dt} = a_1 - b_1 P_T, \quad a_1, b_1 > 0.
\end{equation}
Note that \eqref{eq:24} has an exact finite difference scheme of the form (see \cite{Mickens6})
\begin{equation}\label{eq:25}
\dfrac{P_{T, k + 1} - P_{T, k}}{\phi^*(\Delta t)} = a_1 - b_1P_{T, k}, \quad \phi^*(\Delta t) = \dfrac{1 - e^{-b_1\Delta t}}{b_1}.
\end{equation}
Applying the NSFD scheme \eqref{eq:16} we obtain
\begin{equation}\label{eq:26}
P_{T, k + 1} = P_{T, k} + \dfrac{\phi}{1 + m\phi}(a_1 - b_1P_{T, k}).
\end{equation}
\begin{lemma}\label{lemma4.7}
The following assertions hold for the difference equation \eqref{eq:26}:
\begin{enumerate}
\item the sequence $\{P_{T, k}\}_k$ defined in \eqref{eq:26} monotonically converges to ${a_1}/{b_1}$ as $k \to \infty$, provided that
\begin{equation}\label{eq:27a}
m \geq m_{GCL} := b_1;
\end{equation}
\item The scheme \eqref{eq:26} is an exact finite difference scheme for \eqref{eq:24} if
\begin{equation}\label{eq:27}
\phi(\Delta t) = \phi_E(\Delta t) := \dfrac{\phi^*(\Delta t)}{1 - m\phi^*(\Delta t)},
\end{equation}
where $\phi^*$ is given in \eqref{eq:25}.
\end{enumerate}
\end{lemma}
\begin{proof}
\textbf{Proof of Part (1).} It follows from \eqref{eq:26} that
\begin{equation*}
\begin{split}
P_{T, k + 1} &= P_{T, k} + \dfrac{\phi}{1 + m\phi}(a_1 - b_1P_{T, k}) = \dfrac{a_1\phi}{1 + m\phi} +
\bigg(1 - \dfrac{b_1\phi}{1 + m\phi}\bigg)P_{T, k}\\
%%%
&= \dfrac{a_1\phi}{1 + m\phi} + \bigg(1 - \dfrac{b_1\phi}{1 + m\phi}\bigg)\Bigg[\dfrac{a_1\phi}{1 + m\phi} + \bigg(1 - \dfrac{b_1\phi}{1 + m\phi}\bigg)P_{T, k-1}\Bigg]\\
%%%
&= \dfrac{a_1\phi}{1 + m\phi} + \dfrac{a_1\phi}{1 + m\phi} \bigg(1 - \dfrac{b_1\phi}{1 + m\phi}\bigg) + \bigg(1 - \dfrac{b_1\phi}{1 + m\phi}\bigg)^2P_{T, k-1}.
\end{split}
\end{equation*}
By continuing the above process, we have
\begin{equation*}
P_{T, k + 1} = \dfrac{a_1\phi}{1 + m\phi}\mathlarger{\mathlarger\sum_{i = 0}^k}\bigg(1 - \dfrac{b_1\phi}{1 + m\phi}\bigg)^i + \bigg(1 - \dfrac{b_1\phi}{1 + m\phi}\bigg)^kP_{T, 0}.
\end{equation*}
Since $m \geq b_1$, $1 - \dfrac{b_1\phi}{1 + m\phi} \in (-1, 1)$. This implies that
\begin{equation}
\lim_{t \to \infty}P_{T, k}=  \dfrac{a_1\phi}{1 + m\phi}\bigg(\dfrac{b_1\phi}{1 + m\phi}\bigg)^{-1} = \dfrac{a_1}{b_1}.
\end{equation}
On the other hand, we infer from $m \geq b_1$ that
\begin{equation*}
\dfrac{\partial P_{T, k+1}}{\partial P_{T, k}} = 1 + \dfrac{\phi}{1 +  m\phi}(-b_1) \geq 0,
\end{equation*}
which implies the monotonocity of $\{P_{T, k}\}_k$. The proof is completed.\\
%%%
\textbf{Proof of Part (2).} The exact scheme \eqref{eq:25} can be written in the form
\begin{equation}\label{eq:28}
P_{T, k+1} = P_{T, k} + \phi^*(a_1 - b_1P_{T, k}).
\end{equation}
We infer from \eqref{eq:28} and \eqref{eq:26} that \eqref{eq:26} is an exact scheme for \eqref{eq:24} if
\begin{equation*}
\phi^* = \dfrac{\phi}{1 + m\phi}.
\end{equation*}
This condition is equivalent to \eqref{eq:27}. The proof is complete.
\end{proof}
%%%%%%%%%%%%%%%%%%%%%%%%%%%%%%%%%%%%%%%
From Lemma \ref{lemma4.7} we have:
\begin{theorem}\label{theorem4.8}
The NSFD scheme \eqref{eq:16} satisfies the GCL if one of the following conditions holds:
\begin{enumerate}
\item $m \geq m_{GCL} := b_1$.
\item $\phi(\Delta t) = \phi_E(\Delta t)$, where $\phi_E$ is given in \eqref{eq:27}.
\end{enumerate}
\end{theorem}
\begin{remark}
The function $\phi_E$ given in \eqref{eq:27} satisfies $\phi(\Delta t) = \Delta t + \mathcal{O}(\Delta t^2)$ and is defined  for all $\Delta t > 0$ if $b_1 \geq m$. When $b_1 < m$, it is defined only if $\Delta t \ne -\dfrac{\ln(m - b_1)}{b_1}$.
\end{remark}
\begin{remark}
The sub-conservation law will be satisfied if the DCL and GCL are satisfied.
\end{remark}
\subsection{Dynamics consistency}\label{sub4.4}
For the sake of convenience, we summarize the results constructed in Subsections \ref{sub4.1}-\ref{sub4.4} in the following theorem.
\begin{theorem}[Dynamic consistency]\label{theorem4.10}
The NSFD scheme \eqref{eq:16} is
\begin{itemize}
\item dynamically consistent with the positivity of the model \eqref{eq:1} if the condition \eqref{eq:12} is satisfied;
%\begin{equation}\label{eq:NSFD1}
%m \geq m_P := \max\{\alpha,\,\, 0\};
%\end{equation}
%%%
\item dynamically consistent with the LAS of the model \eqref{eq:1} if the condition \eqref{eq:16c} is satisfied;
%\begin{equation}\label{eq:NSFD2}
%m \geq m_S;
%\end{equation}
%%
\item dynamically consistent with the DCL of the model \eqref{eq:1}  for all the values of $m$;
%\begin{equation*}
%-\infty < m < \infty.
%\end{equation*}
%
\item dynamically consistent the GCL the model \eqref{eq:1} if
\begin{equation}\label{eq:NSFD4}
m \geq m_{GCL} := b_1
\end{equation}
or
\begin{equation}\label{eq:NSFD5}
\phi(\Delta t) = \phi_E(\Delta t),
\end{equation}
\end{itemize}
where $\phi_E$ is given in \eqref{eq:27}.
\end{theorem}
\begin{remark}
\begin{itemize}
\item The positivity, LAS and conservation laws are independent of each other. So, they may not occur simultaneously. Therefore, the threshold parameter $m$ should be selected depending on properties that the model \eqref{eq:1} can have.
\item The determination of $m_P$ in \eqref{eq:12} and $m_{GCL}$ in \eqref{eq:NSFD4} is trivial while it is also easy to write a numerical algorithm to compute $m_S$.
{
\item Since the conditions for the NSFD scheme \eqref{eq:16} to be dynamically consistent with respect to the positiivity, LAS and conservation laws of the model \eqref{eq:1} are set to the parameter $m$, we can choose the denominator function $\phi(\Delta t)$ to be $\Delta t$. This shows the role of the parameter $m$ for the dynamic consistency of the NSFD scheme \eqref{eq:16}.
}
\end{itemize}
\end{remark}
\subsection{A note on the nonstandard Euler scheme}\label{sub4.5}
{As emphasized in Remark \ref{Remark3.4new}, the NSFD scheme \eqref{eq:16} can be considered as a nonstandard Euler method. For the sake of convenience, we now consider a special case of the parameter $m$, namely, $m = 0$. In this case, the NSFD scheme \eqref{eq:16} becomes
\begin{equation}\label{eq:16.1}
y_{k + 1} = y_k + \phi(\Delta t) f(y_k).
\end{equation}
This can be considered as a "weakly" NSFD scheme \cite{Dimitrov1}.\par
%%%
Because of the absence of parameter $m$, conditions for the dynamic consistency of the scheme \eqref{eq:16.1} will be set to the denominator function $\phi$. Below, we will give conditions for the scheme \eqref{eq:16.1} to be dynamically consistent with the model \eqref{eq:1}.}\par
%%%
We infer from \eqref{eq:16.1} that $y_{k + 1} \geq 0$ for all $\phi > 0$ if $\alpha \leq 0$. In this case, the nonstandard Euler scheme is unconditionally positive. Otherwise, if $\alpha > 0$, then it follows from \eqref{eq:2} that
\begin{equation*}
y_{k + 1} = y_k + \phi f(y_k) \geq y_k + (-\alpha y_k) = (1 - \phi\alpha)y_k.
\end{equation*}
Consequently, $y_{k + 1} \geq 0$ if $\phi < (\alpha)^{-1}$. So,  we have the following result on the positivity of the nonstandard Euler scheme.
%%%
\begin{lemma}\label{lemma4.11}
The nonstandard Euler scheme \eqref{eq:16.1} preserves the positivity of the model \eqref{eq:1} for all the values of the step size $\Delta t$ whenever the denominator function satisfies the following condition for all $\Delta t > 0$
\begin{equation}\label{eq:NE1}
\phi(\Delta t) < \phi_P :=
\begin{cases}
&\infty, \,\,\,\quad\qquad \alpha \leq 0,\\
&(\alpha)^{-1}, \qquad \alpha > 0.
\end{cases}
\end{equation}
\end{lemma}
%%%%%%%%%%%
The following result is a direct consequence of the results presented in \cite{Dimitrov1, Dimitrov2}.
\begin{corollary}\label{lemma4.12}
The nonstandard Euler scheme \eqref{eq:16.1} preserves the LAS of the model \eqref{eq:1} for all the values of the step size $\Delta t$ if the denominator function satisfies the following condition for all $\Delta t > 0$
\begin{equation}\label{eq:NE2}
\phi(\Delta t) < \phi_S := \min_{\lambda \in \mathcal{F}}\bigg\{-\dfrac{2Re\lambda}{|\lambda|^2}\bigg\},
\end{equation}
where the set $\mathcal{F}$ is defined in \eqref{eq:16a}.
\end{corollary}
%%%
%%%%%%%%%%%
It is easily seen that the nonstandard Euler scheme always preserves the DCL. On the other hand, when applying the nonstandard Euler scheme for systems satisfying the GLC, we obtain
\begin{equation}\label{eq:30}
\dfrac{P_{T, k + 1} - P_{T, k}}{\varphi(\Delta t)} = a_1 - b_1P_{T, k}.
\end{equation}
Hence, the GLC is preserved if
\begin{equation}\label{eq:31a}
\phi < \phi_{GCL} = (b_1)^{-1}.
\end{equation}
Furthermore, the scheme \eqref{eq:30} becomes an exact scheme for \eqref{eq:24} if 
\begin{equation}\label{eq:31}
\phi(\Delta t) = \dfrac{1 - e^{-b_1\Delta t}}{b_1}.
\end{equation}
%%%%%%%%%%
Hence, we obtain the following assertion.
\begin{corollary}
The nonstandard Euler scheme \eqref{eq:16.1} is always dynamically consistent with respect to the DCL and is dynamically consistent with respect to the GCL if the condition \eqref{eq:31a} or the condition \eqref{eq:31} holds.
\end{corollary}
\begin{remark}
\begin{itemize}
\item {By choosing suitable denominator functions, the nonstandard Euler scheme \eqref{eq:16.1} can preserve the positivity, LAS and conservation laws as the NSFD scheme \eqref{eq:16}. Although the schemes \eqref{eq:16} and \eqref{eq:16.1} are all convergent of order $1$, in the next section we will provide a numerical example in which {errors of Scheme \eqref{eq:16} are better than errors  of Scheme \eqref{eq:16.1}}.}
\item The condition \eqref{eq:31} can conflict with the conditions \eqref{eq:NE1} and \eqref{eq:NE2}.
\end{itemize}
\end{remark}
\section{Numerical experiments}\label{sec5}
In this section, we perform some numerical examples to illustrate the theoretical findings and to show advantages of the proposed NSFD scheme.
\begin{example}[Numerical simulation of single-species population models]
{
Consider single-species population models, which describe  the population size changes in the absence of disease and of maturation delay \cite{Cooke}
\begin{equation}\label{eq:5}
\dfrac{dN}{dt} = B(N)N - dN,
\end{equation}
where $N(t)$ stands for the population size at the time $t$, $d > 0$ is the death rate constant, and
$B(N)N$ is a birth rate function with $B(N)$ has the following properties for all $N \in (0, \infty)$\\
(A1) $B(N) > 0$;\\
(A2) $B(N)$ is continuously differentiable with $B'(N) < 0$;\\
(A3) $B(0^+) > d > B(\infty)$.\\
}
The following functions are  examples of birth functions $B(N)$ that are found in biological systems and satisfy (A1)-(A3) (see \cite{Cooke})\\
%%%%%%%%%%%%%%%%%%%
(B1) $B_1(N) = be^{-aN}$, with $a > 0$ and $b > d$;\\
(B2) $B_2(N) = \dfrac{p}{q + N^n}$, with $p, q, n > 0$ and $p > qd$;\\
(B3) $B_3(N) = \dfrac{A}{N} + c$, with $A > 0$, $d > c > 0$.\\
%%%
It is clear that the model \eqref{eq:5} satisfies the condition \textbf{(C1)} with $\alpha = d$. On the other hand, it possesses two equilibrium points, that are $N^0 = 0$ and $N^* = B^{-1}(d)$. Moreover, $N^0$ is always unstable while $N^*$ is always asymptotically stable. Therefore, the condition \textbf{(C2)} also satisfies. If $B(N) = B_3(N)$, then the model \eqref{eq:5} becomes linear and satisfies the GCL.\par
%%%%%%%%%%%%%%
%%
We consider the model \eqref{eq:5} with 
\begin{equation*}
B(N) = B_1(N) = be^{-aN}N - dN, \quad a = d = 1, \quad b = 12
\end{equation*}
as a test problem for numerical schemes. In this case, the model has a unique equilibrium point 
\begin{equation*}
N^* = -{\ln(d/b)}/{a} \approx  2.4849,
\end{equation*}
which is locally asymptotically stable. Note that the conservation laws are absent in this example.\par
For the NSFD scheme \eqref{eq:16}, we compute the parameter $m$ as
\begin{equation*}
m \geq \max\big\{m_P,\,\,m_S\big\} = \max\{1,\,\,1.2425\} = 1.2425.
\end{equation*}
%%%%%%%%%
For the nonstandard Euler scheme, the condition for the denominator function can be computed as
\begin{equation*}
\phi(\Delta t) < \min\{\phi_P,\,\,\phi_S\} = \min\{1,\,\, 0.8049\} = 0.8049
\end{equation*}
Numerical solutions generated by the NSFD scheme \eqref{eq:16} with $m = 2.5$ and $\Delta t \in \{2.0,\, 0.1\}$ and by the nonstandard Euler scheme with $\phi(\Delta t) = ({1 - e^{-2.5\Delta t}})/{2.5}$ and $\Delta t \in \{1,\, 0.1\}$ are depicted in Figure \ref{Fig:1}. In this figure, each blue curve represents a solution corresponding to a specific initial value and the green arrows represents the evolution of the continuous model. We observe from Figure \ref{Fig:1} that the numerical solutions generated by the NSFD model \eqref{eq:16} are positive and stable. Furthermore, the NSFD model replicates the dynamics of the continuous model and its behaviour does not depend on the chosen step sizes. Similarly, the nonstandard Euler scheme \eqref{eq:16.1} also preserves the positivity and stability of the continuous model regardless of the chosen step sizes.\par
%%%%%%%%%%%%%%%%%%%%%%%
We now use the standard Euler scheme and a well-known {two-stage Runge-Kutta (RK2)} scheme, namely, the explicit trapezoidal scheme  (see \cite{Ascher}) for the purpose of comparing them with the NSFD schemes. Approximations obtained from these standard schemes are sketched in Figure \ref{Fig:2}. It is clear that the standard Euler scheme generates an unstable numerical solution, which oscillates around the equilibrium position. Hence, the dynamics of the continuous model cannot be preserved. The similar unusual phenomena occurs when using the standard RK2 scheme with step sizes $\Delta t = 1$ and $\Delta t = 1.2$. More clearly, the RK2 scheme using $\Delta t = 1.2$ generates the approximation oscillating around the equilibrium point and when $\Delta t = 1.0$, the corresponding numerical solution converges to a spurious equilibrium point. This result is completely consistent with analysis performed by Mickens in \cite{Mickens1, Mickens2, Mickens3, Mickens4, Mickens5}.\par
%%%
Before ending this example, we compute errors obtained from the NSFD scheme \eqref{eq:16}, the nonstandard Euler scheme \eqref{eq:16.1} and the standard Euler scheme. {Note that theses schemes are all convergent of order $1$ since $\phi(\Delta t) = \Delta t + \mathcal{O}(\Delta t^2)$}. For this purpose, we consider $N(0) = 2$ and use the approximation generated by the classical four-stage Runge-Kutta scheme (see \cite{Ascher}) with $\Delta t = 10^{-5}$ as a benchmark (reference) solution since it is impossible to find the exact solution. Errors provided by the three numerical schemes are shown in Tables \ref{tabl1} and \ref{tabl2}. In these tables, $error$ and $error_T$ are computed by
\begin{equation*}
error = \max_{1 \leq k \leq N}|N(t_k) - N_k|, \quad error_T = \sum_{i = 1}^N|N(t_k) - N_k|. 
\end{equation*}
%%%%%%%%%%%%
From Tables \ref{tabl1} and \ref{tabl2}, we see that the errors provided by the NSFD scheme \eqref{eq:16} are better. The errors of the nonstandard Euler scheme are better than the ones provided by the standard Euler scheme. Hence, advantages of the nonstandard schemes in term of errors are shown.
\end{example}
%%%%%%%%%%%%
\newpage
%\begin{figure}[H]
%\centering
%\includegraphics[height=8.5cm,width=14cm]{fig1.eps}
%\caption{Numerical approximations generated  by the NSFD method \eqref{eq:16} with $\Delta t = 2.0$ and $m = 2.5$.}\label{fig:1}
%\end{figure}
%%%%
%\begin{figure}[H]
%\centering
%\includegraphics[height=8.5cm,width=14cm]{fig2.eps}
%\caption{Numerical approximations generated  by the NSFD method \eqref{eq:16} with $\Delta t = 1.0$ and $m = 2.5$.}\label{fig:2}
%\end{figure}
%%%
%\begin{figure}[H]
%\centering
%\includegraphics[height=8.5cm,width=14cm]{fig3.eps}
%\caption{Numerical approximations generated  by the NSFD method \eqref{eq:16} with $\Delta t = 0.1$ and $m = 2.5$.}\label{fig:3}
%\end{figure}
%%%%%%%%%%%
%%
%%
%%Figure  1 new
\begin{figure}[H]
\subfloat[NSFD scheme, $\Delta t = 2.0$]{%
\includegraphics[height=9cm,width=6cm]{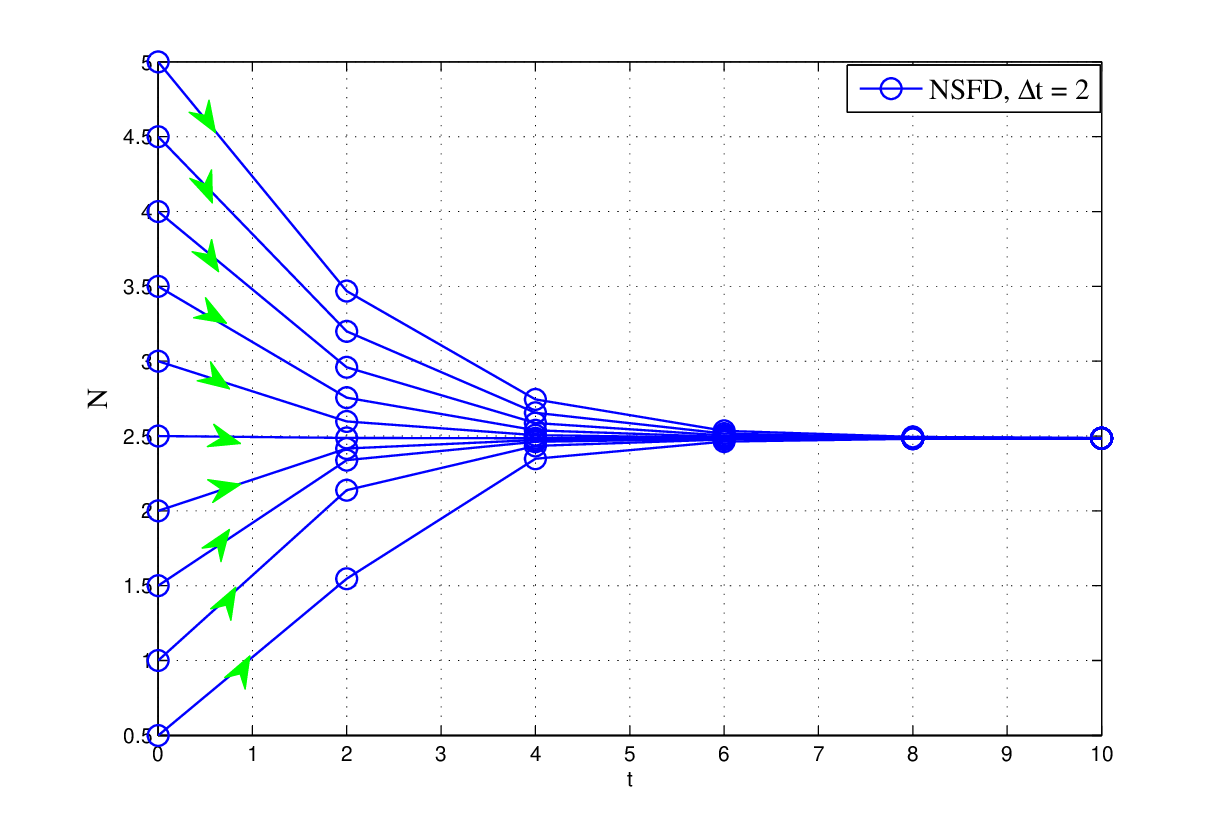}
\label{Figure:1a}
}\hfill
\subfloat[NSFD scheme, $\Delta t = 0.1$]{%
\includegraphics[height=9cm,width=6cm]{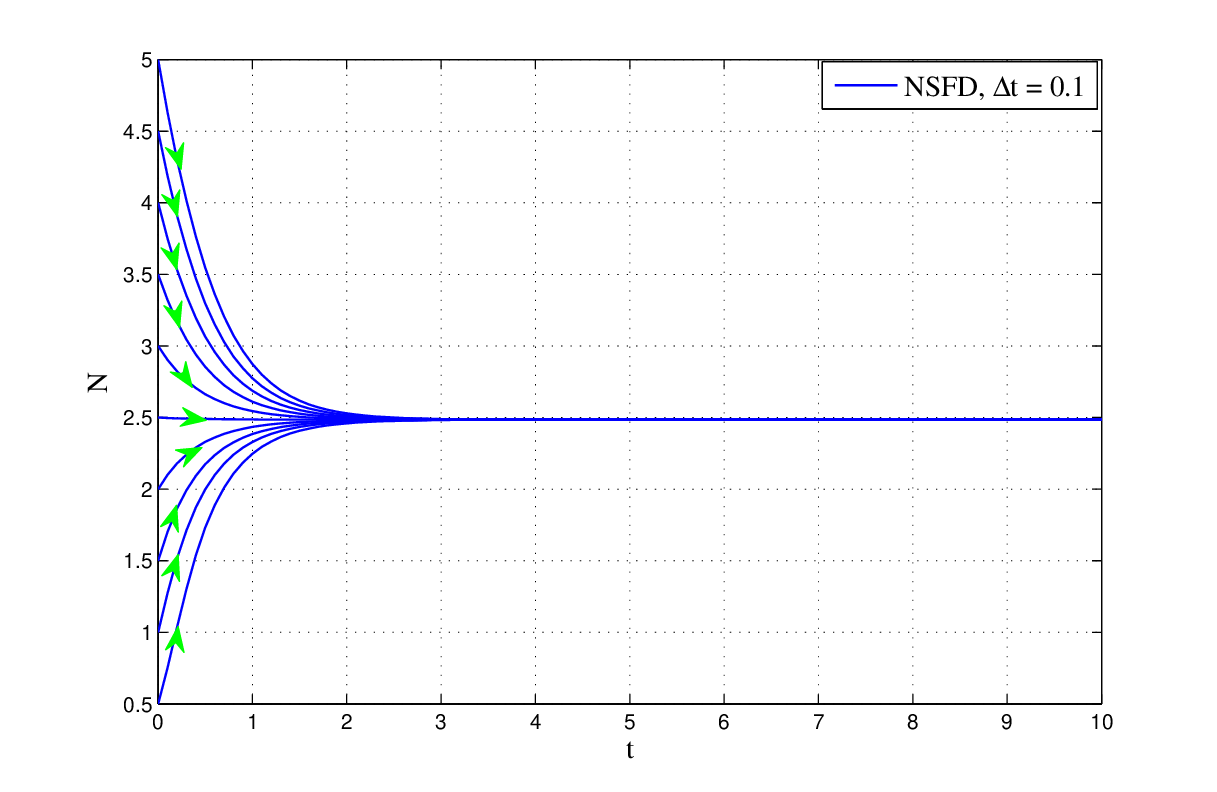}
\label{Figure:1b}
}\hfill
%%%
%%%
\subfloat[Nonstandard Euler scheme, $\Delta t = 1.0$]{%
\includegraphics[height=9cm,width=6cm]{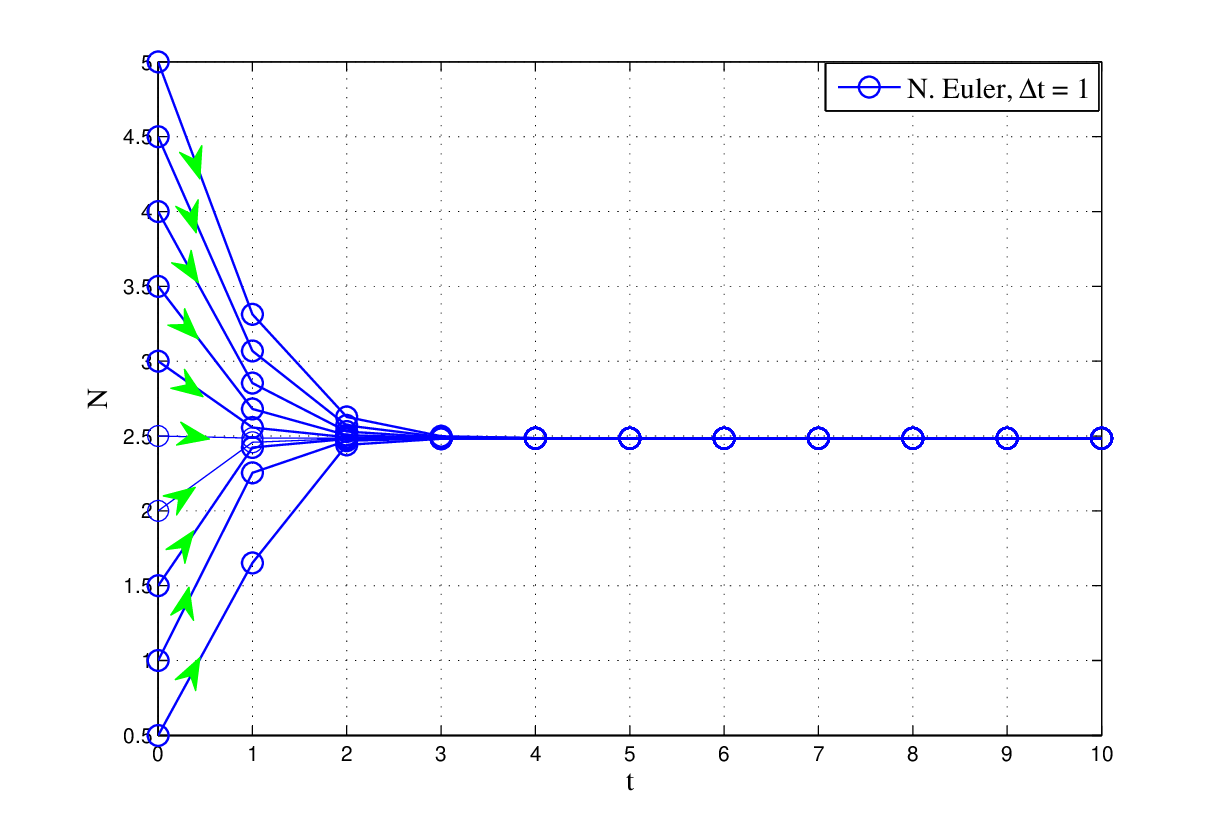}
\label{Figure:1c}
}\hfill
\subfloat[Nonstandard Euler scheme, $\Delta t = 0.1$]{%
\includegraphics[height=9cm,width=6cm]{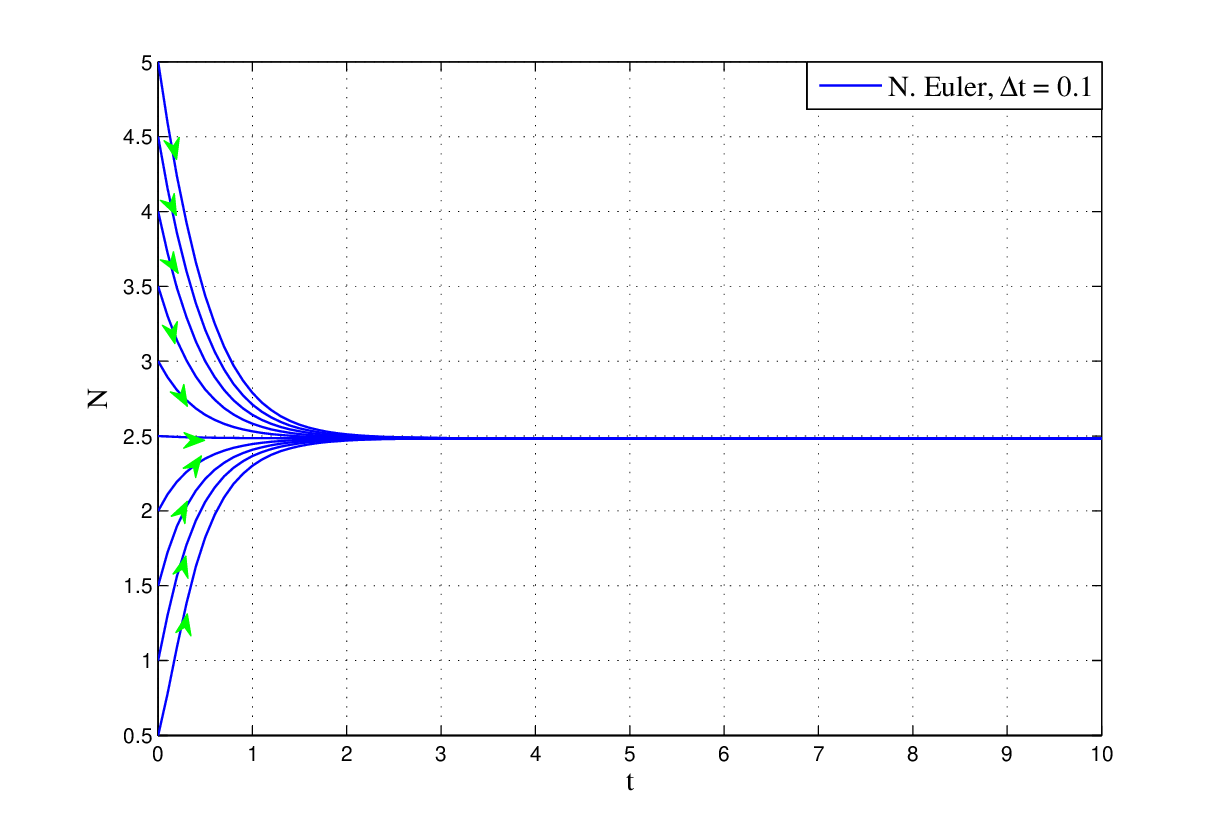}
\label{Figure:1d}
}
\caption{Numerical approximations of the model \eqref{eq:5} generated  by the NSFD scheme \eqref{eq:16} with $m = 2.5$ and $\Delta t \in \{2.0,\, 0.1\}$ and by the nonstandard Euler scheme \eqref{eq:16.1} with $\phi(\Delta t) = (1 - e^{-2.5\Delta t})/2.5$ and $\Delta t \in \{1, 0.1\}$.}
\label{Fig:1}
\end{figure}
\begin{figure}[H]
\subfloat[Euler scheme using $\Delta t = 1.0$.]{%
\includegraphics[height=10cm,width=6cm]{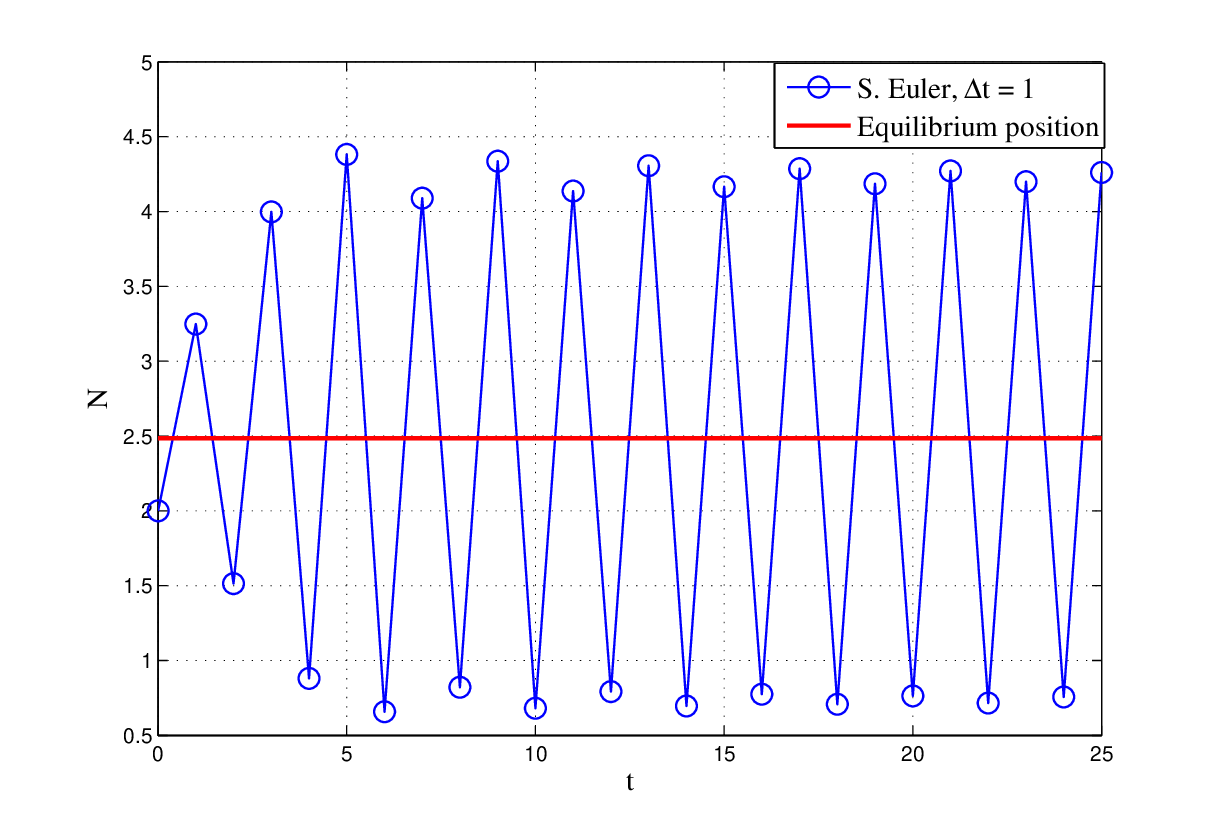}
\label{Figure:2a}
}\hfill
\subfloat[RK2 scheme using $\Delta t = 1.0$ and $\Delta t = 1.2$]{%
\includegraphics[height=10cm,width=6cm]{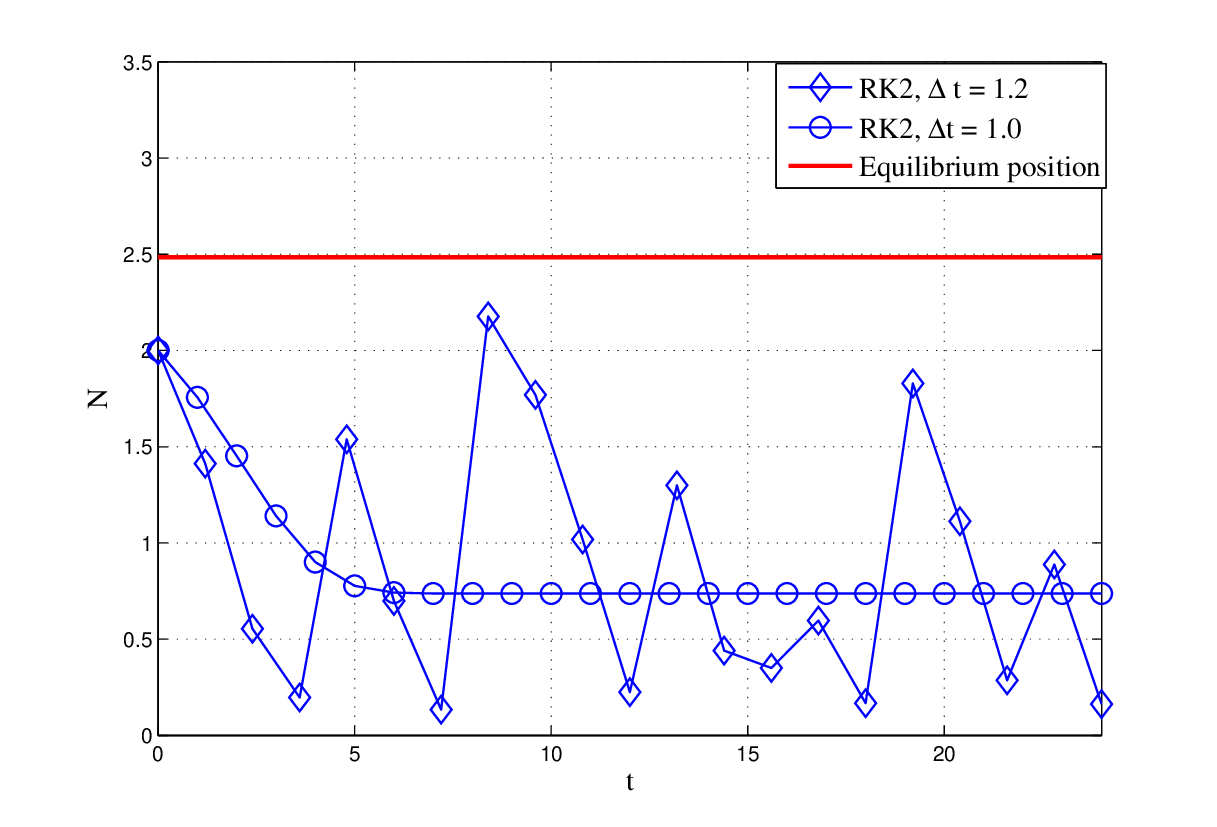}
\label{Figure:2b}
}
\caption{Numerical approximation of the model \eqref{eq:5}  generated  by the standard Euler and RK2 schemes.}
\label{Fig:2}
\end{figure}
%\begin{figure}[H]
%\centering
%\begin{subfigure}{1.0\textwidth}
%\includegraphics[width=\textwidth]{fig6.eps}
%\caption{The standard Euler scheme using $\Delta t = 1.0$.}\label{fig:3a}
%\end{subfigure}
%\vfill
%%%
%%%
%\begin{subfigure}{1.0\textwidth}
%\includegraphics[width=\textwidth]{fig7.eps}
%\caption{The standard RK2 scheme using $\Delta t = 1.0$ and $\Delta t = 1.2$.}\label{fig:3b}
%\end{subfigure}
%\caption{Numerical approximation of the model \eqref{eq:5}  generated  by the standard Euler and RK2 schemes.}\label{fig:3}
%\end{figure}
%%%%%%%%%%%%%
%%%%%%%%%%%%%%%%%%%%%%%%%%
%\begin{figure}[H]
%\centering
%\includegraphics[height=8.5cm,width=14cm]{fig4.eps}
%\caption{Numerical approximations generated  by the nonstandard Euler scheme \eqref{eq:16.1} with $\Delta t = 1$ and $\phi(\Delta t) = (1 - e^{2.5\Delta t})/2.5$.}\label{fig:4}
%\end{figure}
%%%
%%%
%\begin{figure}[H]
%\centering
%\includegraphics[height=8cm,width=14cm]{fig5.eps}
%\caption{Numerical approximations generated  by the nonstandard Euler scheme \eqref{eq:16.1} with $\Delta t = 0.1$ and $\phi(\Delta t) = (1 - e^{-2.5\Delta t})/2.5$.}\label{fig:5}
%\end{figure}
%%%
%%%
%\begin{figure}[H]
%\centering
%\includegraphics[height=8cm,width=14cm]{fig6.eps}
%\caption{Numerical approximation generated  by the standard Euler scheme using $\Delta t = 1$.}\label{fig:6}
%\end{figure}
%%%%%%%%%%%%%
%\begin{figure}[H]
%\centering
%\includegraphics[height=8cm,width=14cm]{fig7.eps}
%\caption{Numerical approximations generated  by the RK2 scheme using $\Delta t = 1$ and $\Delta t = 1.2$.}\label{fig:7}
%\end{figure}
%%%%
%%%%%%%%
%%%%%%%%%
\begin{table}[H]
\begin{center}
\caption{The values of the $error$ of the NSFD scheme \eqref{eq:16} with $m_s = 1.5$, the nonstandard Euler scheme with $\phi(h) = (1 - e^{-1.5\Delta t})/1.5$ and the standard Euler scheme}\label{tabl1}
\begin{tabular}{ccccccccccccccccccccc}
\hline
$\Delta t$&NSFD $error$&N. Euler $error$&S. Euler $error$\\
\hline
$1$&$0.053229803837052$&$0.200391028658851$&$1.895818024330166$\\ 
%%%%%%%%%%%%%%
%%%%%%%%%%%%%%
$10^{-1}$&$0.002547712779198$&$0.010718671314132$&$0.025728835277831$\\
%%%%%%%%%%%%%%%%%%%%%
%%%%%%%%%%%%%%%%%%%%%
%%%%%%%%%%%%%%%%%%%%%
$10^{-2}$&$3.474294297345359e-004$&$9.794823732640623e-004$&$0.002320869273385$\\
%%%%%%%%%%%%
%%%%%%%%%%%%
$10^{-3}$&$3.568404980169859e-005$&$ 9.709160784732163e-005$&$2.298145180823497e-004$\\
%%%%%%%%%%
%%%%%%%%%%
$10^{-4}$&$3.577852512925972e-006$&$9.700655808853043e-006$&$2.295902635651714e-005$\\
%%%%%%%%%%%%%%
%%%%%%%%%%%%
$10^{-5}$&$3.578797715952931e-007$&$9.699803174001431e-007$&$2.295678681374369e-006$\\
%%%
\hline
\end{tabular}
\end{center}
\end{table}
%%%%%%%%%%%%%%%%%%%%%
%%%%%%%%%%%%%%%%%%%%%%%%%%%%%%%%%%%%%
\begin{table}[H]
\begin{center}
\caption{The values of the $error_T$ of the NSFD scheme \eqref{eq:16} with $m_s = 1.5$, the nonstandard Euler scheme with $\phi(h) = (1 - e^{-1.5\Delta t})/1.5$ and the standard Euler scheme}\label{tabl2}
\begin{tabular}{ccccccccccccccccccccc}
\hline
$\Delta t$&NSFD $error_T$&N. Euler $error_T$&S. Euler $error_T$\\
\hline
$1$&$0.056862533841534$&$0.257861518971072$&$15.530365282265436$\\ 
%%%%%%%%%%%%%%
%%%%%%%%%%%%%%
$10^{-1}$&$0.030063654051754$&$ 0.109862865600443$&$ 0.256967636704650$\\
%%%%%%%%%%%%%%%%%%%%%
%%%%%%%%%%%%%%%%%%%%%
%%%%%%%%%%%%%%%%%%%%%
$10^{-2}$&$0.039734990275286$&$0.103390564043856$&$0.247233541153000$\\
%%%%%%%%%%%%
%%%%%%%%%%%%
$10^{-3}$&$ 0.040703612013705$&$0.102744510888389$&$0.246264374931792$\\
%%%%%%%%%%
%%%%%%%%%%
$10^{-4}$&$0.040800485618428$&$0.102679929285270$&$0.246167506109657$\\
%%%%%%%%%%%%%%
%%%%%%%%%%%%
$10^{-5}$&$0.040810166884441$&$0.102673437145564$&$0.246157825305468$\\
%%%
\hline
\end{tabular}
\end{center}
\end{table}
%%%%%%%%%%%%%%%%%%%%
%%%%%%%%%%%%%%%%%%%%%%%%%%%%%%%%%%%%%%
%
\begin{example}[Numerical simulation of a predator-prey system]
{Consider predator-prey model with a Beddington-DeAngelis functional response and linear intrinsic growth rate of the prey population, which is derived from a general predator-prey model proposed by DeAngelis et al in \cite{DeAngelis}. The model under consideration is given by
\begin{equation}\label{eq:6}
\begin{split}
&\dfrac{dx}{dt} = x - \dfrac{Axy}{1 + x + y},\\
%%%
&\dfrac{dy}{dt} = \dfrac{Exy}{1 + x + y} - Dy,
\end{split}
\end{equation}
where $x$ and $y$ are the prey and predator population sizes, respectively, and all the parameters are assumed to be positive. We refer the readers to \cite{Dimitrov4} for details of this model.\par
%%%
%%%
Clearly, the model \eqref{eq:6} satisfies the condition (\textbf{C1)} with $\alpha = \max\{A - 1, D\}$. On the other hand, stability analysis in \cite{Dimitrov4} implies that \textbf{(C2)} holds if $A \ne E$. Numerical simulation of the model \eqref{eq:6} by NSFD schemes was also examined in \cite{Dimitrov2}.\par
%%%%%%%%%%%%%
Clearly, the model \eqref{eq:6} satisfies the condition (\textbf{C1)} with $\alpha = \max\{A - 1, D\}$. On the other hand, stability analysis in \cite{Dimitrov4} implies that \textbf{(C2)} holds if $A \ne E$. Numerical simulation of the model \eqref{eq:6} by NSFD schemes was also examined in \cite{Dimitrov2}.\par
}
We now consider the predator-prey model \eqref{eq:6} with the following parameters
\begin{equation*}
A = 6, \quad D = 5, \quad E = 7.
\end{equation*}
In this case, the model has a unique positive equilibrium point $(x^*, y^*) = (6, \,\,1.4)$ and the Jacobian matrix evaluating at $(x^*, y^*)$ is
\begin{equation*}
J(x^*, y^*) =
\begin{pmatrix}
0.7143&   -3.5714\\
&\\
0.3333&   -0.8333
\end{pmatrix}.
\end{equation*}
Consequently, its eigenvalues are
\begin{equation*}
\lambda_{1, 2} =  -0.0595 \pm 0.7692i,
\end{equation*}
which implies that $(x^*, y^*)$ is locally asymptotically stable.\par
By simple calculations, we determine the condition for the parameter $m$ of the NSFD scheme \eqref{eq:16} as
\begin{equation*}
m \geq \max\{m_P, m_S\} = 5.
\end{equation*}
Therefore, we can take $m = 5.1$.\par
%%%
Numerical approximations provided by the Euler and RK2 schemes and the NSFD scheme \eqref{eq:16} using $(x(0), y(0)) = (4.5, 1)$ are depicted in Figure  \ref{Fig:3}. We see that the phase plane generated by the Euler scheme  is unstable and spirals out from the equilibrium position; the RK2 scheme provides an approximation which oscillates around the equilibrium point; meanwhile, the approximations obtained from the NSFD scheme \eqref{eq:16} correctly preserves the dynamics of the predator-prey model.\par
Note that the condition for the denominator function $\phi$ of the nonstandard Euler scheme \eqref{eq:16.1} is
\begin{equation*}
\phi < \min\{\phi_P, \phi_S\} = \dfrac{1}{5}.
\end{equation*}
Hence, we can take $\phi(\Delta t) = (1 - e^{-\tau \Delta t})/\tau$, where $\tau > 5$. Then, the dynamic consistency of the nonstandard Euler scheme is guaranteed.
%%%%%%%
%Figure 4
%
%\begin{figure}[H]
%\centering
%\begin{subfigure}{1.0\textwidth}
%\includegraphics[width=\textwidth]{fig15.eps}
%\caption{The Euler scheme using $\Delta t = 0.25$ after $400$ iterations.}\label{fig:4a}
%\end{subfigure}
%\vfill
%%%
%%%
%\begin{subfigure}{1.0\textwidth}
%\includegraphics[width=\textwidth]{fig16.eps}
%\caption{The RK2 scheme using $\Delta t = 1.25$ after $800$ iterations.}\label{fig:4b}
%\end{subfigure}
%\caption{Numerical approximations of the model \eqref{eq:6} generated  by the standard Euler and RK2 schemes.}\label{fig:4}
%\end{figure}
%%%%%%%%%%%%%%%%
%%%%%%%%%%%%%%%%%%
%Figure 5
%\begin{figure}[H]
%\centering
%\begin{subfigure}{1.1\textwidth}
%\includegraphics[width=\textwidth]{fig17.eps}
%\caption{$\Delta t = 1.25$ and $800$ iterations.}\label{fig:5a}
%\end{subfigure}
%\vfill
%%%
%%%
%\begin{subfigure}{1.1\textwidth}
%\includegraphics[width=\textwidth]{fig18.eps}
%\caption{$\Delta t = 0.1$ and $10^5$ iterations.}\label{fig:5b}
%\end{subfigure}
%\caption{Numerical approximations of the model \eqref{eq:6} generated by the NSFD scheme \eqref{eq:16}.}\label{fig:5}
%\end{figure}
%%
%Figure 3 new
\begin{figure}[H]
\subfloat[Euler scheme using $\Delta t = 0.25$ and $400$ iterations]{%
\includegraphics[height=9cm,width=6cm]{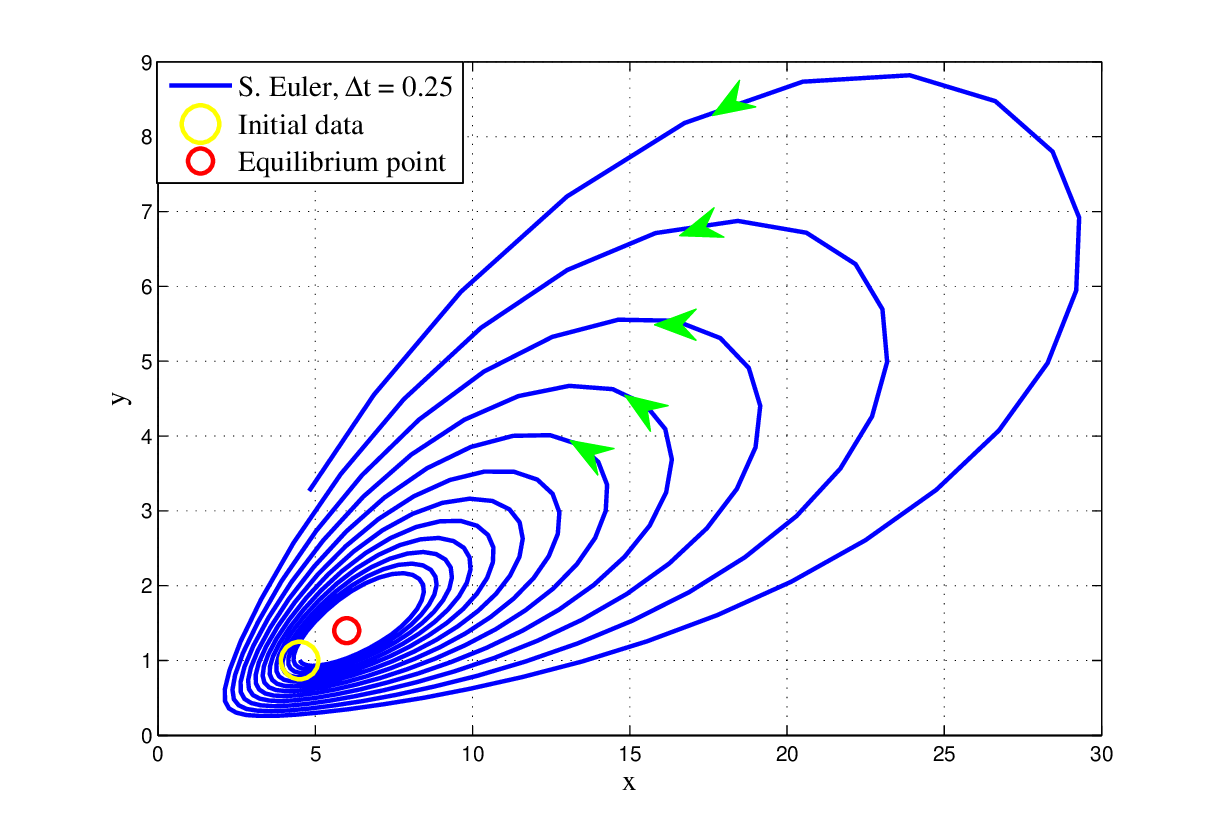}
\label{Figure:3a}
}\hfill
\subfloat[RK2 scheme using $\Delta t = 1.25$ and $800$ iterations.]{%
\includegraphics[height=9cm,width=6cm]{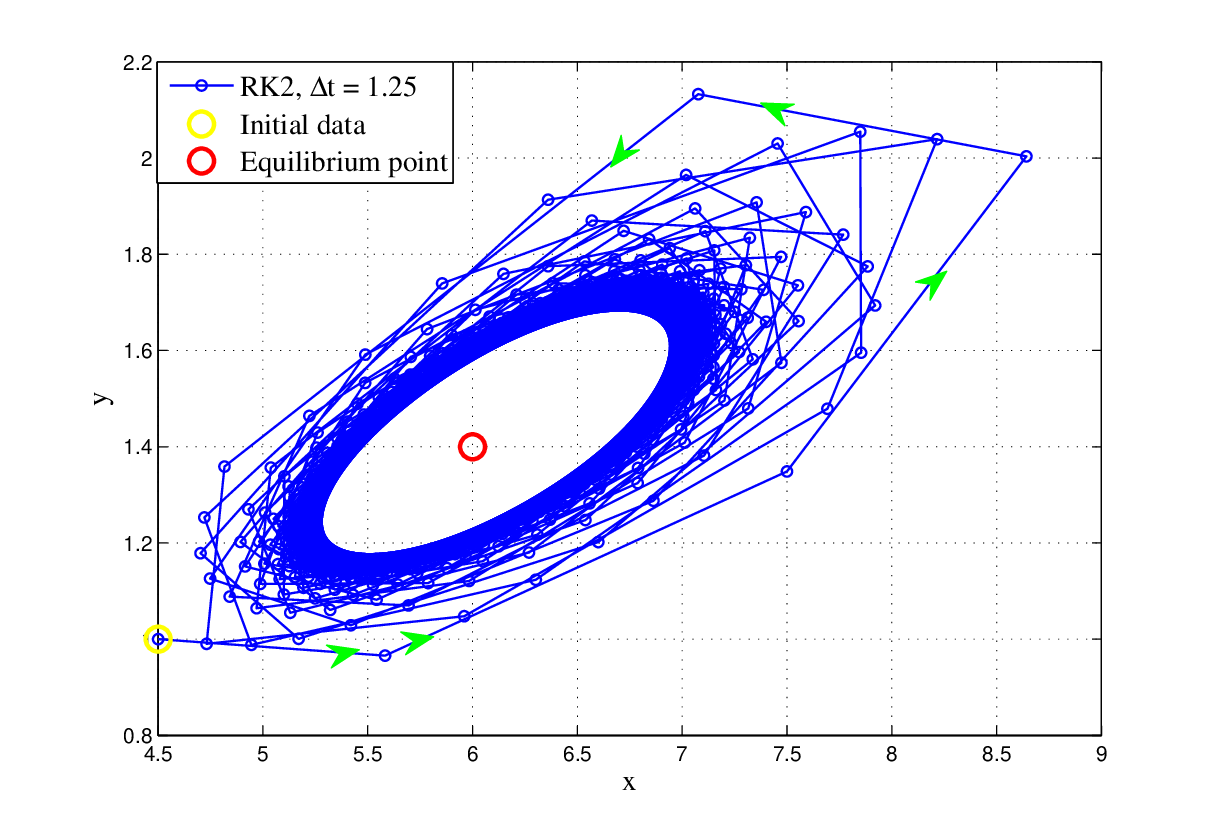}
\label{Figure:3b}
}\hfill
%%%
%%%
\subfloat[NSFD scheme using $\Delta t = 1.25$ and $800$ iterations]{%
\includegraphics[height=9cm,width=6cm]{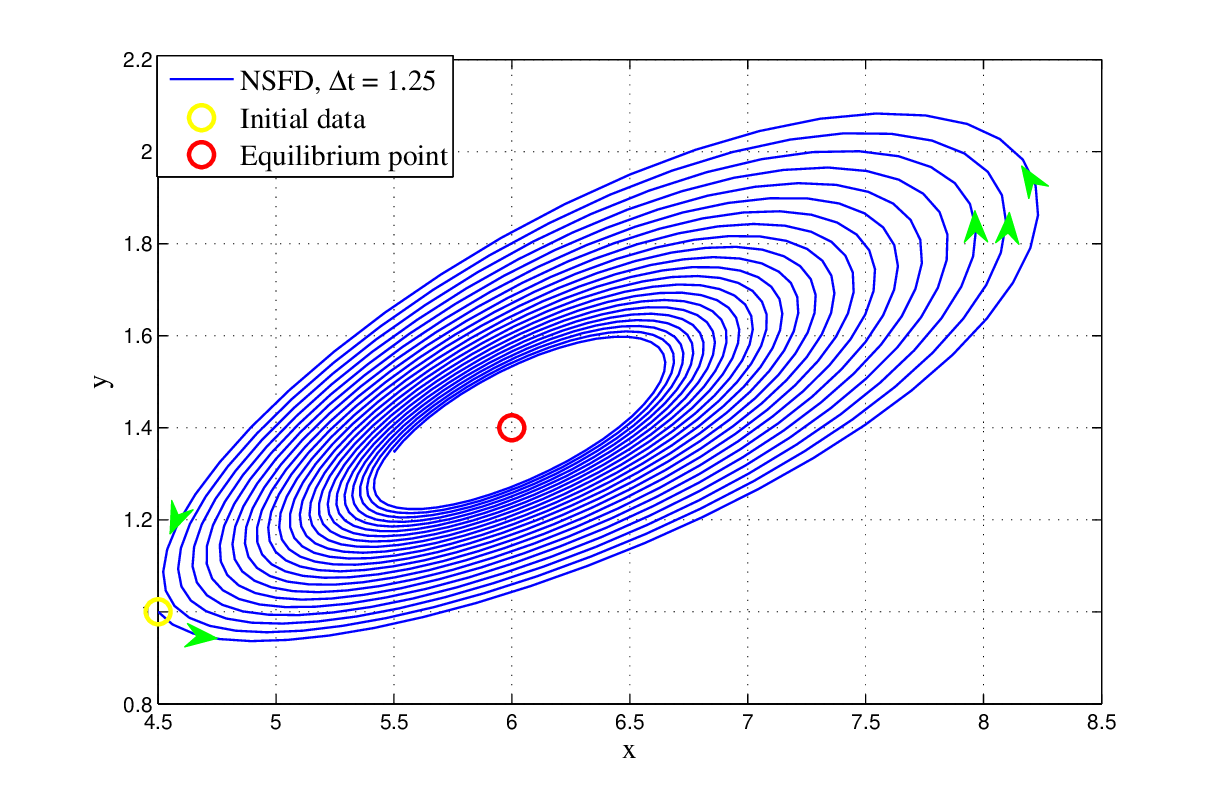}
\label{Figure:3c}
}\hfill
\subfloat[NSFD scheme using $\Delta t = 0.1$ and $10^5$ iterations]{%
\includegraphics[height=9cm,width=6cm]{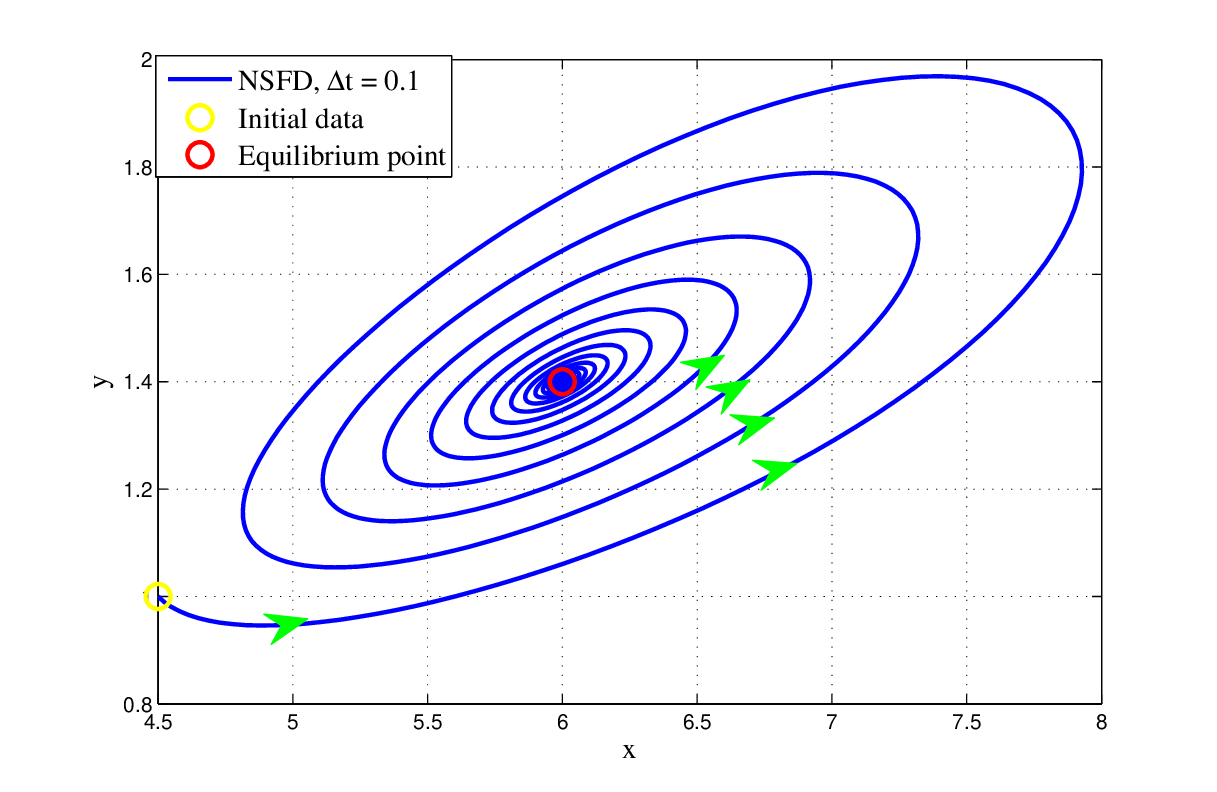}
\label{Figure:3d}
}
\caption{Numerical approximations of the model \eqref{eq:6} generated by the NSFD scheme \eqref{eq:16} and the standard Euler and RK2 schemes.}
\label{Fig:3}
\end{figure}
%%%%%%%%%%%%%%%%%%%%%%%
%%%
%\begin{figure}[H]
%\centering
%\includegraphics[height=8cm,width=14cm]{fig15.eps}
%\caption{Numerical approximation generated by the Euler scheme with $\Delta t = 0.25$ after $400$ iterations.}\label{fig:15}
%\end{figure}
%%%
%%%
%%%%
%\begin{figure}[H]
%\centering
%\includegraphics[height=8cm,width=14cm]{fig16.eps}
%\caption{Numerical approximation generated by the RK2 scheme with $\Delta t = 1.25$ after $800$ iterations.}\label{fig:16}
%\end{figure}
%%
%%%
%\begin{figure}[H]
%\centering
%\includegraphics[height=8cm,width=14cm]{fig17.eps}
%\caption{Numerical approximation generated by the NSFD scheme \eqref{eq:16} with $\Delta t = 1.25$ after $800$ iterations.}\label{fig:17}
%\end{figure}
%%%
%%%%
%\begin{figure}[H]
%\centering
%\includegraphics[height=8cm,width=14cm]{fig18.eps}
%\caption{Numerical approximation generated by the NSFD scheme \eqref{eq:16} with $\Delta t = 0.1$ after $10^5$ iterations.}\label{fig:18}
%\end{figure}
\end{example}
%
%
%%
%
%%%%%
%%%%%%
%%%%%%
\begin{example}[Numerical simulation of a vaccination model]
{
%\subsection{A vaccination model with non-linear incidence}\label{subsec3.4}
Consider an infectious disease with a preventive (prophylactic) vaccine and an effective therapeutic treatment proposed by Gumel and Moghadas in \cite{Gumel1}. This model monitors the temporal dynamics of susceptible individuals ($S$), vaccinated individuals ($V$) and infected individuals $(I)$ and is described by 
\begin{equation}\label{eq:8}
\begin{split}
&\dfrac{dS}{dt} = \Pi - \dfrac{c\beta_1I}{1 + I}S - \xi S + \alpha^* I - \mu S,\\
%%%
&\dfrac{dV}{dt} = \xi S - \dfrac{c\beta_2I}{1 + I}V - \mu V,\\
%%%
&\dfrac{dI}{dt} = \dfrac{c\beta_1I}{1 + I}S + \dfrac{c\beta_2I}{1 + I}V - \alpha^* I - \mu I.
\end{split}
\end{equation}
Here, all the parameters are positive. Mathematical formulation and qualitative study for the model \eqref{eq:8} were performed in \cite{Gumel1}.\par
%%%
It is easy to check that the model \eqref{eq:8} satisfies the condition \textbf{(C1)} with $\alpha = \max\{c\beta_1 + \xi + \mu, c\beta_2 + \mu, \alpha^* + \mu\}$. Meanwhile, the condition \textbf{(C2)} holds if the basic reproduction number (BRN) of the model satisfies
\begin{equation*}
\mathcal{R}_0 = \dfrac{c\beta_1\Pi}{(\mu + \alpha^*)(\mu + \xi)} + \dfrac{c\beta_2\Pi\xi}{\mu(\mu + \alpha^*)(\mu + \xi)} \ne 1.
\end{equation*}
Also, the GCL is satisfied because
\begin{equation*}
\dfrac{d(S + V + I)}{dt} = \Pi - \mu(S + V + I).
\end{equation*}
}
In this example, we consider the vaccination model \eqref{eq:8} with the following parameters
\begin{equation*}
\begin{split}
&\Pi = 700, \quad \beta_1 = 10^{-4}, \quad \beta_2 = 10^{-6}, \quad \mu = 0.03 \qquad (\mbox{see \cite{Gumel1}}),\\
&\alpha^* = 0.95, \qquad c = 8, \qquad \xi = 0.95, \qquad\qquad\qquad\quad\,\, (\mbox{assumed}).
\end{split}
\end{equation*}
In this case, the BRN is $\mathcal{R}_0 =  0.7677 < 1$ and the model has a unique disease-free (uninfected) equilibrium point $E^0 = (S^0, V^0, I^0) = (714.2857,\,\,  2.2619 \times 10^4,\,\, 0)$, which is locally asymptotically stable (see \cite{Gumel1}).\par
We now compute the parameter $m$ for the NSFD scheme \eqref{eq:16}. First, it is easy to see that
\begin{equation*}
m_P = \max\{c\beta_1 + \xi + \mu, c\beta_2 + \mu, \alpha + \mu\} = 0.98.
\end{equation*}
Since the Jacobian matrix of the model evaluating at $E^*$ has three eigenvalues, that are
\begin{equation*}
\lambda_1 = -0.98, \quad \lambda_2 = -0.03, \quad \lambda_3 =  -0.23,
\end{equation*}
we obtain
\begin{equation*}
m_S = 0.49. 
\end{equation*}
Note that the model satisfying a GCL, namely,
\begin{equation*}
\dfrac{d(S + V + I)}{dt} = \Pi - \mu(S + V + I).
\end{equation*}
Consequently, the condition for the preservation of the GCL is
\begin{equation*}
m \geq m_{GCL} = {\mu} = 0.03.
\end{equation*}
Therefore, the condition for the parameter $m$ is
\begin{equation*}
m \geq \max\{m_P, m_S, m_{GCL}\} = 0.98.
\end{equation*}
In a similar way, we can determine the condition for the nonstandard Euler scheme as
\begin{equation*}
\phi(\Delta t) = \dfrac{1 - e^{-\tau\Delta t}}{\tau}, \quad \tau > (0.98)^{-1} \approx 1.02.
\end{equation*}
\end{example}
Numerical solutions obtained from the standard Euler, RK2 and NSFD scheme are depicted in Figure \ref{Fig:4}, respectively. In these figures, each blue curve represents a phase space corresponding to a specific initial value, the green arrows show the evolution of the continuous model and the red circle indicates the position of the stable equilibrium point.\\
We observe from Figures \ref{Figure:4a} and \ref{Figure:4b} that the standard Euler scheme generates a negative and unstable numerical approximation, which oscillates around the equilibrium point with an increasing amplitude; the solution provided by the standard RK2 scheme is not positive and converges to a spurious equilibrium point; conversely, the NSFD scheme correctly preserves the dynamics of the continuous model regardless of the chosen step sizes (see Figures \ref{Figure:4c} and \ref{Figure:4d}).\par
Note that the vaccination model satisfies a GCL. For this reason, we now consider approximations of the total population generated by the Euler, RK2 and NSFD scheme \eqref{eq:16}. From numerical approximations  in Figure \ref{Fig:5}, we see that the NSFD schemes preserves the GCL but this property is not preserved by the standard schemes.
\newpage
%
%Figure 6
%\begin{figure}[H]
%\centering
%\begin{subfigure}{1.0\textwidth}
%\includegraphics[width=\textwidth]{fig8.eps}
%\caption{The Euler scheme with $(S(0), V(0), I(0)) = (500, 2000, 100)$ and $\Delta t = 2$ after $100$ iterations.}\label{fig:6a}
%\end{subfigure}
%%%
%%%
%\begin{subfigure}{1.0\textwidth}
%\includegraphics[width=\textwidth]{fig9.eps}
%\caption{The RK2 scheme with some different initial data and $\Delta t = 2$ after $100$ iterations.}\label{fig:6b}
%\end{subfigure}
%\caption{Numerical approximations of the model \eqref{eq:8} generated by the standard Euler and RK2 schemes.}\label{fig:6}
%\end{figure}
%%
%%
%%
%%
%%Figure 7
%\begin{figure}[H]
%\centering
%\begin{subfigure}{1.0\textwidth}
%\includegraphics[width=\textwidth]{fig10.eps}
%\caption{$\Delta t = 2.0$ after $250$ iterations.}\label{fig:7a}
%\end{subfigure}
%%%
%%%
%\begin{subfigure}{1.0\textwidth}
%\includegraphics[width=\textwidth]{fig11.eps}
%\caption{$\Delta t = 0.1$ after $5000$ iterations.}\label{fig:7b}
%\end{subfigure}
%\caption{Phase spaces of the model \eqref{eq:8} generated  by the NSFD scheme with some different initial data.}\label{fig:7}
%\end{figure}
%%%%%%%%%%%%%%
%%
%
%Figure 4 new
\begin{figure}[H]
\subfloat[Euler scheme using $(S(0), V(0), I(0)) = (500, 2000, 100)$ and $\Delta t = 2$ after $100$ iterations]{%
\includegraphics[height=9cm,width=6cm]{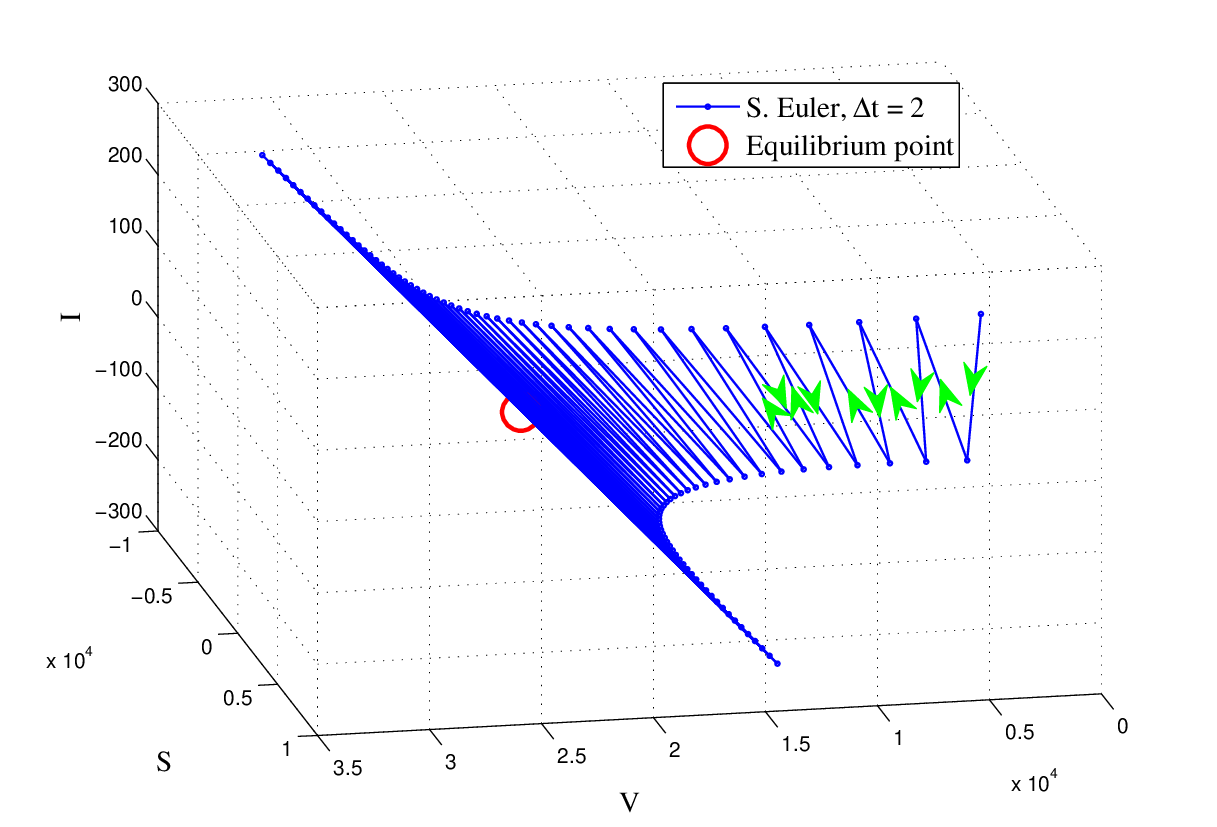}
\label{Figure:4a}
}\hfill
\subfloat[RK2 scheme using $\Delta t = 2$ after $100$ iterations]{%
\includegraphics[height=9cm,width=6cm]{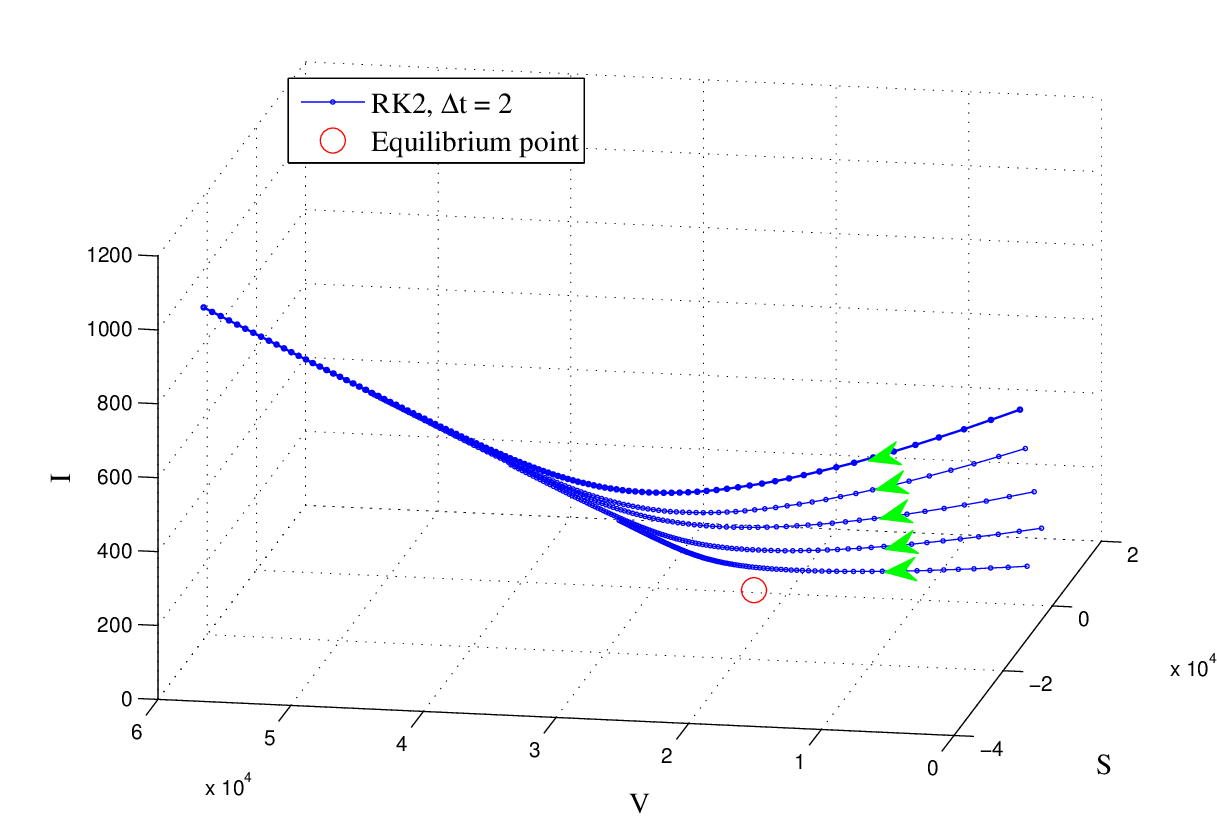}
\label{Figure:4b}
}\hfill
%%%
%%%
\subfloat[NSFD scheme with $\Delta t = 2.0$ after $250$ iterations]{%
\includegraphics[height=9cm,width=6cm]{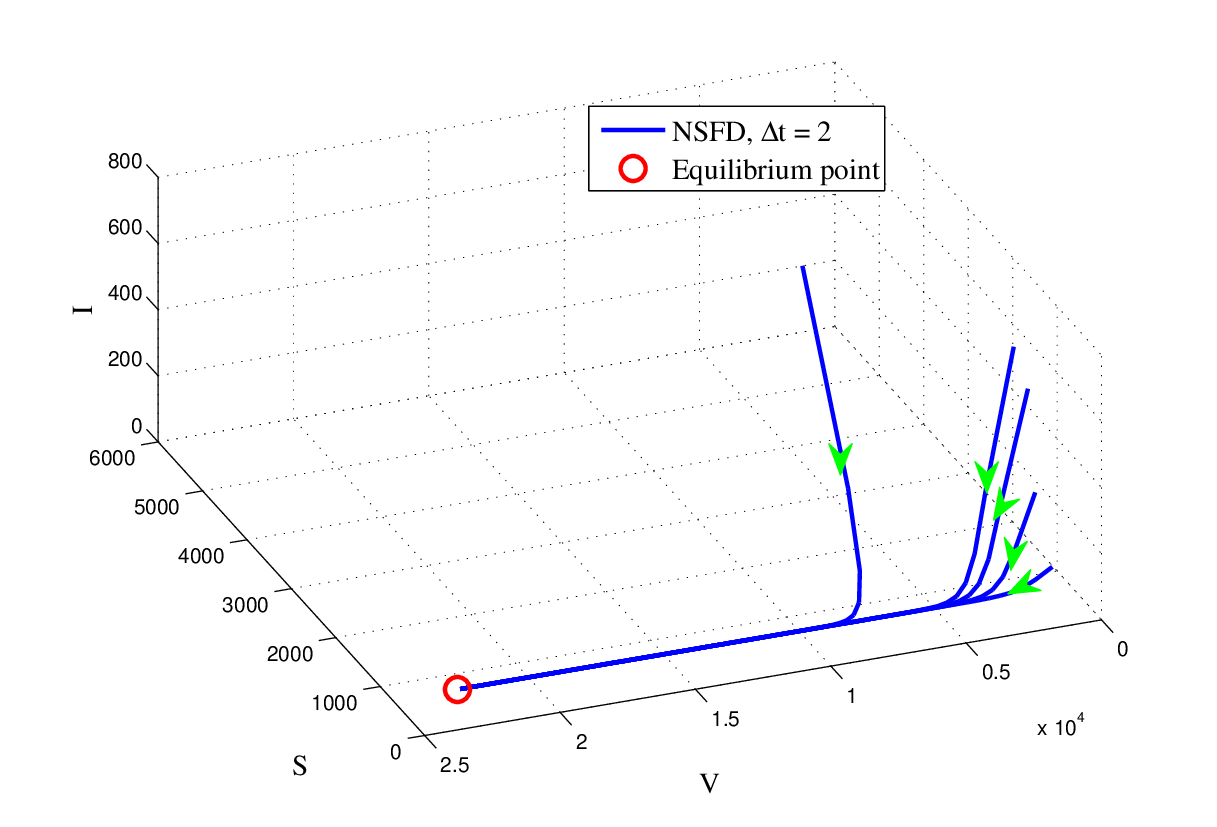}
\label{Figure:4c}
}\hfill
\subfloat[NSFD scheme with $\Delta t = 0.1$ after $5000$ iterations]{%
\includegraphics[height=9cm,width=6cm]{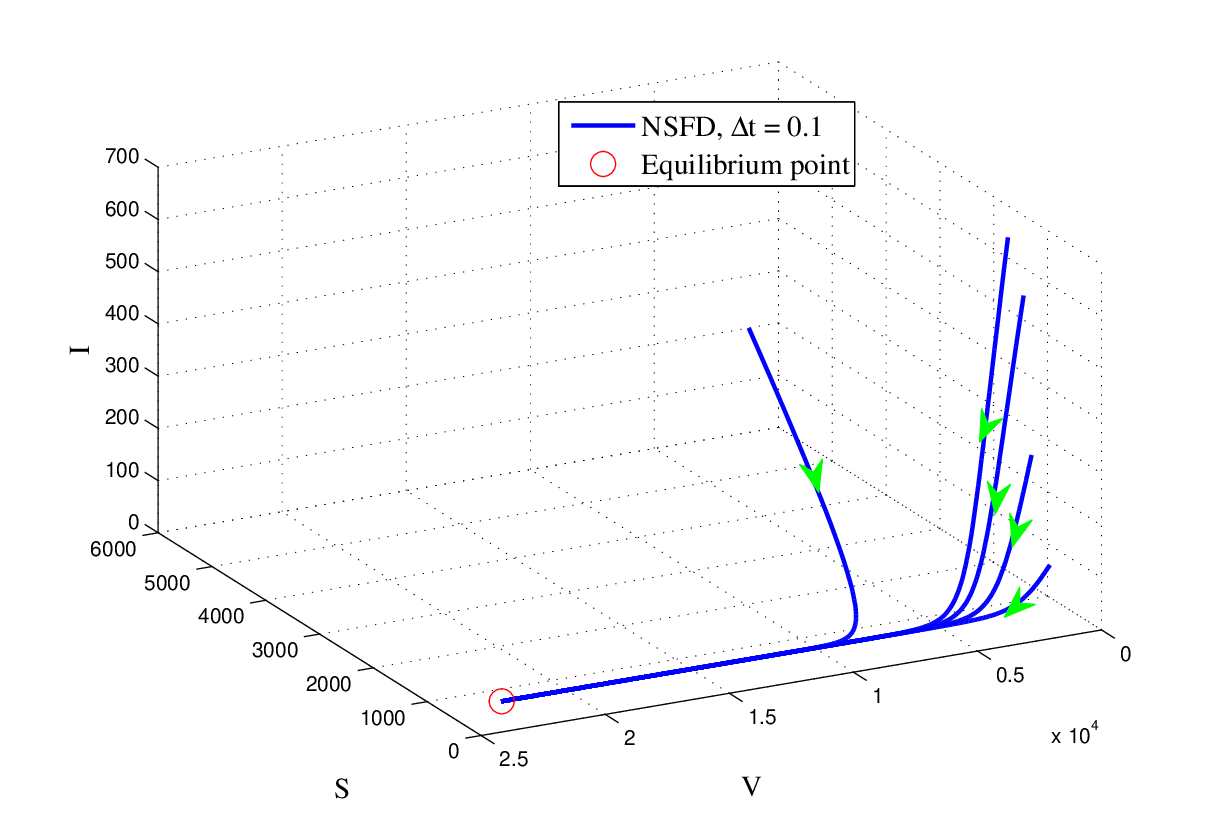}
\label{Figure:4d}
}
\caption{Numerical approximations of the model \eqref{eq:8} generated by the standard Euler and RK2 schemes and by the NSFD scheme \eqref{eq:16}.}
\label{Fig:4}
\end{figure}
%%
%Figure 5 new
\begin{figure}[H]
\subfloat[Euler scheme using $\Delta t = 2.5$]{%
\includegraphics[height=9cm,width=6cm]{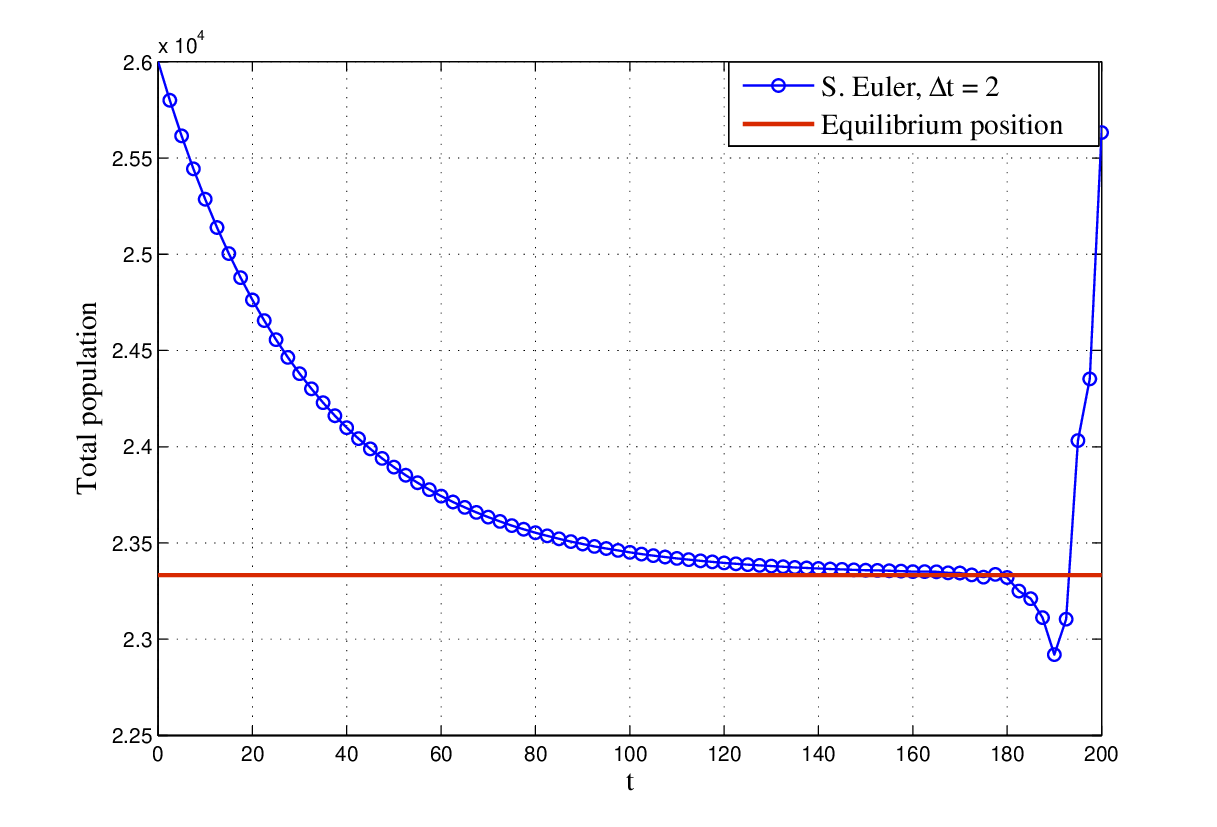}
\label{Figure:5a}
}\hfill
\subfloat[RK2 scheme using $\Delta t = 2.5$]{%
\includegraphics[height=9cm,width=6cm]{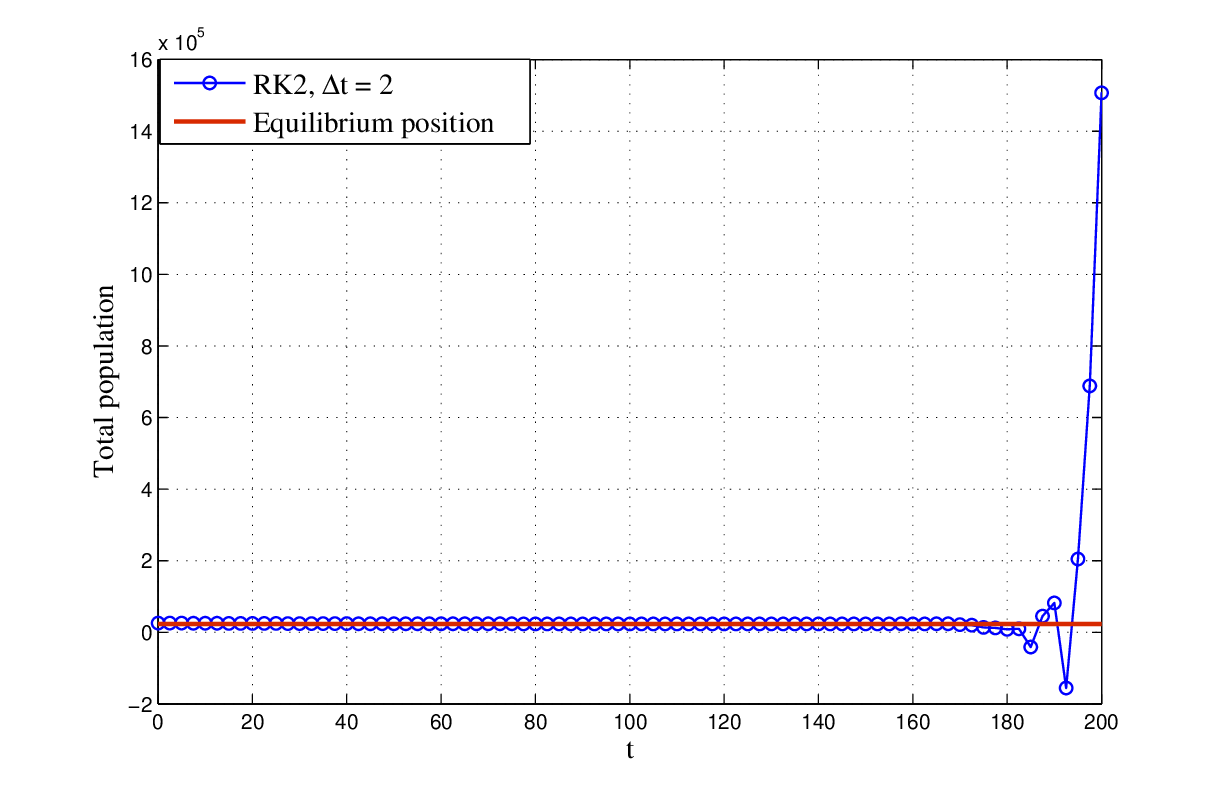}
\label{Figure:5b}
}\hfill
%%%
%%%
\subfloat[NSFD scheme using $\Delta t = 2.5$ and $\Delta t = 5.0$]{%
\includegraphics[height=9cm,width=12cm]{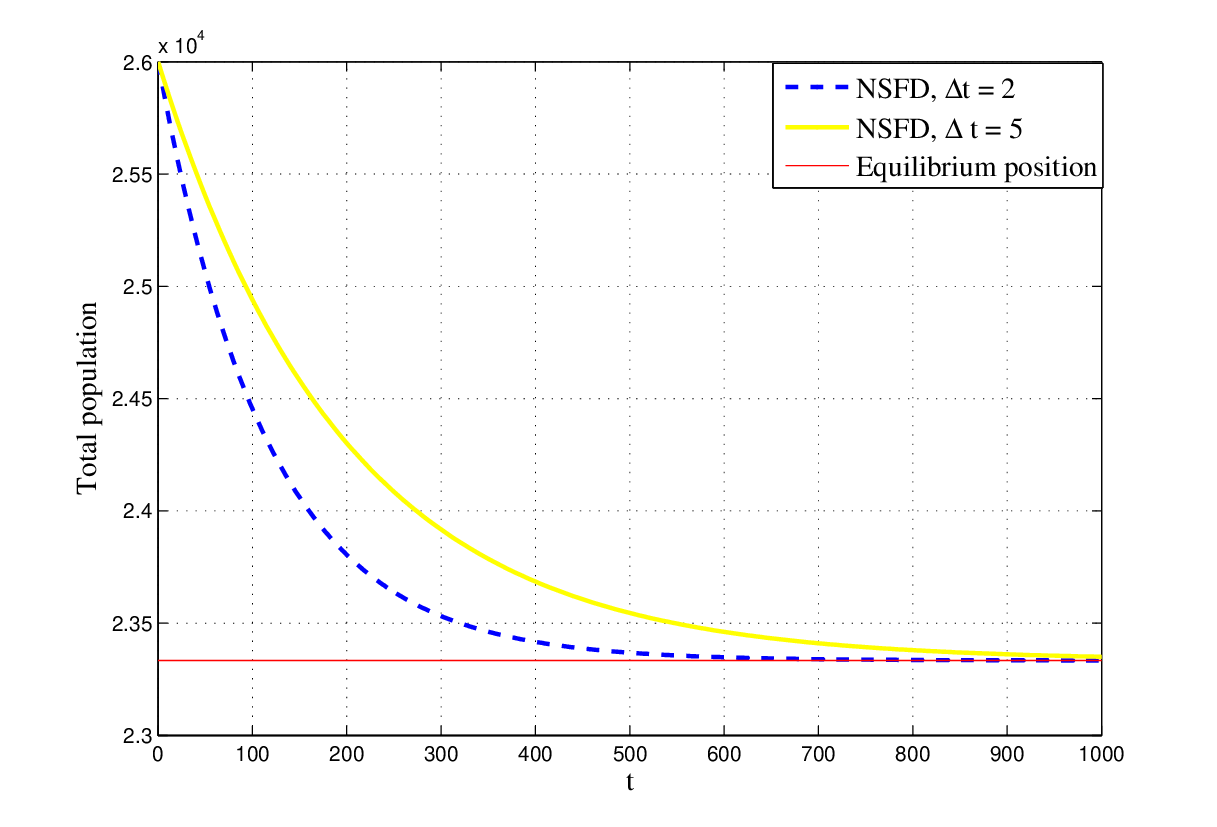}
\label{Figure:5c}
}
\caption{Approximations of the total population generated by the standard Euler and RK2 schemes and the NSFD scheme \eqref{eq:16} using $(S(0),\, V(0),\, I(0)) = (5000,\, 20000,\, 1000)$.}
\label{Fig:5}
\end{figure}

\section{Conclusions, discussions and some open problems}\label{sec6}
In this work, we have applied {the NSFD methodology proposed by Mickens} to construct a generalized NSFD method for a class of autonomous dynamical systems that satisfy positivity, stability and conservation laws.   This NSFD method is constructed based on a novel non-local approximation for the right-side functions of the dynamical systems.  By rigorous mathematical analyses, we have shown that the NSFD method is dynamically consistent with respect to the positivity, asymptotic stability and three classes of conservation laws, including direct conservation, generalized conservation and sub-conservation laws.
Additionally, a set of numerical experiments has been also performed to illustrate the theoretical findings and to show advantages of the proposed NSFD method over some well-known standard numerical schemes. The numerical results are not only consistent with the theoretical assertions but also indicate that the NSFD method is easy to be implemented and can be used to solve a broad range of mathematical models arising in real-life.\par
%%%%%%%%%%%%%%%%
It worth noting that the threshold parameters $(m_S, m_P, m_{GCL})$ for the NSFD method \eqref{eq:16} and $(\phi_P, \phi_S, \phi_{GCL})$ for the nonstandard Euler method \eqref{eq:16.1} can be computed easily. More clearly, these threshold parameters can be obtained by simple calculations or simple numerical algorithms.\par
%%%%%%%%%%%%%%%%%
Similarly to the NSFD methods formulated in \cite{Anguelov1, Dimitrov1, Dimitrov2}, the NSFD method \eqref{eq:16} also provides a stability-preserving numerical scheme for autonomous dynamical systems in general contexts. However, the positivity of the NSFD schemes in \cite{Anguelov1, Dimitrov1, Dimitrov2} is not guaranteed automatically while the NSFD method \eqref{eq:16} has the ability to preserve not only the positivity but also the conservation laws of the system \eqref{eq:1}.\par
%%%%%%%%%%%%%%%%
In \cite{Wood1},  Wood and Kojouharov introduced a nonstandard numerical method preserving the positivity and LAS of autonomous dynamical systems. This method is computationally efficient, easy
to be implemented and can be used to solve a broad range of problems in science and engineering. However, it requires the determination of the sign of the right-hand side functions at each iteration, and hence, the computational time is increased. For the NSFD method \eqref{eq:16}, we only need to compute the parameter $m$ exactly once and then use it for all iterations. Furthermore, the computation of $m$ is easy and can be done by a simple algorithm.\par
%%%%%%%%%%%%%%%%
Before ending this section, we mention some open problems and future directions that research in NSFD methods can take.
%%%%
\begin{enumerate}
\item {The NSFD method \eqref{eq:16} requires the determination of the number $\alpha$ in the condition \eqref{eq:2}. Although $\alpha$ can be determined easily for most of the models in literature, it is necessary to construct an algorithm for computing $\alpha$ in general cases. Even though many important mathematical models in real-word applications satisfy the condition \eqref{eq:2}, this condition leads to a big restriction of the application of the proposed NSFD method. Hence, it is important to study extensions of the NSFD method \eqref{eq:16} for mathematical models having the positivity and asymptotic stability but not {satisfying Condition \eqref{eq:2}}}.\par
\item NSFD methods preserving the asymptotic stability of non-hyperbolic equilibrium points: Most stability-preserving NSFD methods including the NSFD method \eqref{eq:16} require equilibrium points of dynamical systems to be hyperbolic. {The construction of generalized NSFD methods preserving the LAS of dynamical systems even when they have non-hyperbolic equilibrium points  has been studied in some previous work (see, for instance, \cite{Anguelov5})} but is still an important open problem.
%%%
\item Global asymptotic stability (GAS) of NSFD methods: It is well-known that analyzing the GAS of NSFD schemes for mathematical models having the GAS property  is very important but not a simple task in general. {This problem has been mentioned and considered in many previous works (see, for instance \cite{Al-Kahby, Anguelov2, DangHoang1, DangHoang2, DangHoang3, DangHoang4, DangHoang5, Hoang2, Roeger2, Roeger3})}. We observe from the numerical examples in Section \ref{sec5} that the NSFD method can preserve not only the LAS but also the GAS of the continuous models under consideration. Therefore, it is suitable to predict that there exists a positive number $m_{GAS}$ that plays as a GAS threshold of the NSFD method \eqref{eq:16}, i.e., the NSFD method \eqref{eq:16} preserves the GAS of the system \eqref{eq:1} whenever $m \geq m_{GAS}$. This prediction will be useful in studying the construction of NSFD methods preserving the GAS of dynamical systems.
\item High-order NSFD methods: High-order NSFD methods for differential equations have attracted the attention of many researchers with various approaches (see, for example \cite{Chen-Charpentier, DangHoang, Gonzalez-Parra, Hoang, Kojouharov, Martin-Vaquero1, Martin-Vaquero2}); hence, the improvement accuracy of the NSFD method \eqref{eq:16} is essential with many practice applications. A simple approach for improving accuracy of the NSFD method \eqref{eq:16} is to combine it with Richardson's extrapolation method or variable step size strategies. This is easy to do because \eqref{eq:16} is simple and fully explicit.\par
%%%%%%%%%%%%%
%%%%%%%%%%%%%%%
\item New NSFD methods for differential equation models arising in real-world applications: Although NSFD methods have been strongly developed and achieved many important results, fields of theory and practice always pose new and complex problems; hence, constructing effective NSFD methods for these problems is very essential. Recently, NSFD methods for fractional differential equations have been widely studied (see, for instance \cite{Arenas, Cresson1}). It was proved that NSFD methods for fractional-order systems are efficient and powerful; however, except for the positivity, other properties of fractional-order NSFD methods have not been widely discussed. It is easy to extend the NSFD method \eqref{eq:16} for systems in the context of fractional derivative operators. Then, the asymptotic stability in particular and qualitative study in general for the fractional-order NSFD schemes will be posed. It is appropriate to predict that the parameter $m$ under suitable conditions will guarantee the dynamic consistency of the proposed NSFD schemes.
\end{enumerate}
%%%%%%%%%%

\section*{Acknowledgments}
We would like to thank the editor and anonymous referees for useful and valuable comments that led to a
great improvement of the paper.

\bibliographystyle{amsalpha}

\begin{thebibliography}{999}
%A
\bibitem{Adamu}
E. M. Adamu, K. C. Patidar, and A. Ramanantoanina,
\textit{An unconditionally stable nonstandard finite difference method to solve a mathematical model describing Visceral Leishmaniasis}, Math. Comput. Simul. \textbf{187} (2021) 171-190.
%
%
\bibitem{Adekanye}
O. Adekanye, and T. Washington, 
\textit{Nonstandard finite difference scheme for a Tacoma Narrows Bridge model},
Appl. Math. Model. \textbf{62} (2018) 223-236.
%
%
\bibitem{Al-Kahby}
H. Al-Kahby, F. Dannan and  S. Elaydi,
\textit{Non-standard Discretization Methods for Some Biological Models},
Applications of Nonstandard Finite Difference Schemes, pp. 155-180 (2000).
%
%
\bibitem{Allen}
L. J. S. Allen,
\textit{Introduction to Mathematical Biology},
Pearson Education, Inc, Prentice Hall, 2007.
%
%
\bibitem{Amarasekare}
P. Amarasekare and H. Possingham,
\textit{Patch dynamics and metapopulation theory: the case of successional species},
J. Theor. Biol. \textbf{209} (2001)333-344.
%
%
\bibitem{Anguelov}
R. Anguelov and J. M. -S. Lubuma,
\textit{Nonstandard finite difference method by nonlocal approximation}, 
Math. Comput. Simul. \textbf{61} (2003) 465-475.
%
%
\bibitem{Anguelov1}
R. Anguelov and J. M.-S. Lubuma,
\textit{Contributions to the mathematics of the nonstandard finite difference method and applications}, Numer. Methods Partial Differ. Equ. \textbf{17} (2001) 518-543.
%%
%%
\bibitem{Anguelov2}
R. Anguelov, Y. Dumont, J. M. -S. Lubuma, and M. Shillor, 
\textit{Dynamically consistent nonstandard finite difference schemes for epidemiological models},
J. Comput. Appl. Math. \textbf{255} (2014) 161-182.
%
%
\bibitem{Anguelov3}
R. Anguelov, J. M.-S. Lubuma, and S. K. Mahudu, 
\textit{Qualitatively stable finite difference schemes for advection-reaction equations}, 
J. Comput. Appl. Math.  \textbf{158} (2003) 19-30.
%
{
\bibitem{Anguelov4}
R. Anguelov, P. Kama, and J. M. -S. Lubuma,
\textit{On non-standard finite difference models of reaction-diffusion equations}, 
J. Comput. Appl. Math.  \textbf{175} (2005) 11-29.
%%
\bibitem{Anguelov5}
R. Anguelov, J. M. -S. Lubuma, and M. Shillor, 
\textit{Topological dynamic consistency of non-standard finite difference schemes for dynamical systems}, 
J. Differ. Equ. Appl.   \textbf{17} (2011) 1769-1791.
}
%%
\bibitem{Arenas}
A. J. Arenas, G. Gonzalez-Parra, and B. M. Chen-Charpentier, 
\textit{Construction of nonstandard finite difference schemes for the SI and SIR epidemic models of fractional order}, 
Math. Comput. Simul. \textbf{121}(2016) 48-63. 
%%
\bibitem{Ascher}
U. Ascher and L. Petzold, 
\textit{Computer methods for ordinary differential equations and differential-algebraic equations}, Society for Industrial and Applied Mathematics, Philadelphia, 1998.
%B
%
%
\bibitem{Brauer}
F. Brauer and C. Castillo-Chavez,
\textit{Mathematical Models in Population Biology and Epidemiology}, Second Edition, Springer New York, NY, 2012.
%
%
\bibitem{Capasso}
V. Capasso and G. Serio,
\textit{A generalization of the Kermack-McKendrick deterministic epidemic model}, 	Math. Biosci. \textbf{42} (1978) 43-61.
%%
\bibitem{Chapwanya}
M. Chapwanya, J. M.-S. Lubuma, and R. E. Mickens, 
\textit{Positivity-preserving nonstandard finite difference schemes for cross-diffusion equations in biosciences}, Comput. Math. Appl. \textbf{68} (2014) 1071-1082.
%%
%
\bibitem{Chen-Charpentier}
B. M. Chen-Charpentier, D. T. Dimitrov, and H. V. Kojouharov,
\textit{Combined nonstandard numerical methods for ODEs with polynomial right-hand sides}, Math. Comput. Simul. \textbf{73} (2006) 105-113.
%
%
\bibitem{Cooke}
K. Cooke,  P. van den Driessche, and X. Zou,
\textit{Interaction of maturation delay and nonlinear birth in population and epidemic models}, J. Math. Biol. \textbf{39} (1999) 332-352.
%
%
\bibitem{Cresson}
J. Cresson and F. Pierret, 
\textit{Non standard finite difference scheme preserving dynamical properties}, J. Comput. Appl. Math. \textbf{303} (2016) 15-30.
%%
%
\bibitem{Cresson1}
J. Cresson and A. Szafra\'nska,
\textit{Discrete and continuous fractional persistence problems – the positivity property and applications}, Commun. Nonlinear Sci. Numer. Simul. \textbf{44} (2017) 424-448.
%%
%%%
%%%%%%%%%%%%%%%%%%%%%%%%%%
\bibitem{DangHoang}
Q. A. Dang and M. T. Hoang,
\textit{Positive and elementary stable explicit nonstandard Runge-Kutta methods for a class of autonomous dynamical systems}, 
Int. J. Comput. Math. \textbf{97} (2020) 2036-2054.
%%
%%
\bibitem{DangHoang1}
Q. A. Dang and M. T. Hoang,
\textit{Nonstandard finite difference schemes for a general predator-prey system}, J. Comput. Sci. \textbf{36} (2019) 101015.
%%
%%
\bibitem{DangHoang2}
Q. A. Dang and M. T. Hoang,
\textit{Complete Global Stability of a Metapopulation Model and Its Dynamically Consistent Discrete Models}, Qual. Theory Dyn. Syst. \textbf{18} (2019) 461-475.
%%
%%
\bibitem{DangHoang3}
Q. A. Dang, M. T. Hoang, Lyapunov direct method for investigating stability of nonstandard finite difference schemes for metapopulation models, J. Differ. Equ. Appl. \textbf{24} (2018) 15-47.
\bibitem{DangHoang4}
Q. A. Dang and M. T. Hoang, 
\textit{Positivity and global stability preserving NSFD schemes for a mixing propagation model of computer viruses}, J. Comput. Appl. Math.  \textbf{374} (2020) 112753.
%%
%%
\bibitem{DangHoang5}
Q. A. Dang and M. T. Hoang,
\textit{Dynamically consistent discrete metapopulation model}, 	J. Differ. Equ. Appl.  \textbf{22} (2016) 1325-1349.
%%
\bibitem{DeAngelis}
D. L. DeAngelis, R. A. Goldstein, and R. V. O'Neill, 
\textit{A Model for Tropic Interaction}, Ecology \textbf{56} (1975) 881-892.
%
%
\bibitem{Dimitrov1}
D. T. Dimitrov and H. V. Kojouharov, 
\textit{Nonstandard finite-difference schemes for general two-dimensional autonomous dynamical systems}, 	Appl. Math. Lett. \textbf{18} (2005) 769-774.
%%
%%
\bibitem{Dimitrov2}
D. T. Dimitrov and H. V. Kojouharov,
\textit{Stability-Preserving Finite-Difference Methods for General Multi-Dimensional Autonomous Dynamical Systems}, Int. J. Numer. Anal. Model. \textbf{4} (2007) 280-290.
%%%
%%
\bibitem{Dimitrov3}
D. T. Dimitrov and H. V. Kojouharov, 
\textit{Dynamically consistent numerical methods for general productive-destructive systems}, J. Differ. Equ. Appl. \textbf{17} (2011) 1721-1736.
%%
\bibitem{Dimitrov4}
D. T. Dimitrov and H. V. Kojouharov, 
\textit{Complete mathematical analysis of predator-prey models with linear prey growth and Beddington-DeAngelis functional response}, Comput. Appl. Math. \textbf{162} (2005) 523-538.
%%%%
%
%E
\bibitem{Edelstein-Keshet}
L. Edelstein-Keshet,\textit{ Mathematical Models in Biology}, Society for Industrial and Applied Mathematics, Philadelphia, 2005.
%%%%%%%%%%
\bibitem{Ehrhardt}
M. Ehrhardt and R. E.Mickens, \textit{A nonstandard finite difference scheme for convection-diffusion equations having constant coefficients}, Appl. Math. Comput.  \textbf{219} (2013) 6591-6604.
%%
\bibitem{Elaydi}
S. Elaydi, \textit{An Introduction to Difference Equations}, Springer New York, 2005.
%F
\bibitem{Fatoorehchi}
H. Fatoorehchi and M. Ehrhardt, 
\textit{Numerical and semi-numerical solutions of a modified Th\'evenin model for calculating terminal voltage of battery cells},  J Energy Storage \textbf{145}(2022) 103746.
%%
\bibitem{Garba}
S. M. Garba, A. B. Gumel, A. S. Hassan, and J. M. -S. Lubuma,
\textit{Switching from exact scheme to nonstandard finite difference scheme for linear delay differential equatio}n,  Appl. Math. Comput. \textbf{258} (2015) 388-403.
%%%
%%
\bibitem{Gonzalez-Parra}
G. Gonz\'alez-Parra, A. J. Arenas, and B. M. Chen-Charpentier, 
\textit{Combination of nonstandard schemes and  Richardson’s extrapolation to improve the numerical solution of population model}s, 
Math Comput Model \textbf{52} (2010) 1030-1036.
%
%
\bibitem{Gumel1}
A. B. Gumel and S. M. Moghadas, 
\textit{A qualitative study of a vaccination model with non-linear incidence}, Appl. Math. Comput. \textbf{143} (2003) 409-419.
%
\bibitem{Hethcote}
W. H. Hethcote, M. Zhien, and  L. Shengbing,
\textit{Effects of quarantine in six endemic models for infectious diseases}, 
Math. Biosci. \textbf{180} (2002) 141-160.
%
\bibitem{Hoang}
M. T. Hoang,
\textit{A novel second-order nonstandard finite difference method for solving one-dimensional autonomous dynamical systems}, Commun. Nonlinear Sci. Numer. Simul. \textbf{114} (2022) 106654.
%%
\bibitem{Hoang1}
M. T. Hoang, 
\textit{Positivity and boundedness preserving nonstandard finite difference schemes for solving Volterra’s population growth model}, Math. Comput. Simul. \textbf{199} (2022) 359-373.
%%
\bibitem{Hoang2}
M. T. Hoang,
\textit{Dynamically consistent nonstandard finite difference schemes for a virus-patch dynamic model},  J Appl Math Comput \textbf{68} (2022) 3397-3423.
%
\bibitem{Hoang3}
M. T. Hoang, Reliable approximations for a hepatitis B virus model by nonstandard numerical schemes, Math. Comput. Simul. \textbf{193}(2022) 32-56.
\bibitem{Hoang5}
M. T. Hoang, Dynamical analysis of a generalized hepatitis B epidemic model and its dynamically consistent discrete model,   Math. Comput. Simul. \textbf{205}(2023) 291-314
%
%%
\bibitem{Horvath}
Z. Horv\'ath,
\textit{On the positivity step size threshold of Runge-Kutta methods},  Appl. Numer. Math. \textbf{53} (2005) 341-356.
%
\bibitem{Kermack1}
W. O. Kermack and A. G. McKendrick,
\textit{Contributions to the mathematical theory of epidemics - I},   Proc. R. Soc. Lond. \textbf{115} (1927) 700-721.
\bibitem{Kermack2}
W. O. Kermack and A. G. McKendrick, 
\textit{Contributions to the mathematical theory of epidemics - II},   Proc. R. Soc. Lond. \textbf{138} (1932) 55-83.
%
\bibitem{Keymer}
J. E. Keymer, P.A. Marquet, J. X. Velasco-Hernandez, and S. A. Levin, 
\textit{Extinction thresholds and metapopulation persistence in dynamic landscapes}, Am. Nat. \textbf{156} (2000) 478-494.
\bibitem{Khalil}
H. K. Khalil,
\textit{ Nonlinear systems}, Third Edition, Prentice Hall, 2002.
%K
\bibitem{Kojouharov}
H.  V. Kojouharov, S. Roy, M. Gupta, F. Alalhareth, and J. M.Slezak,
\textit{A second-order modified nonstandard theta method for one-dimensional autonomous differential equations}, 	Appl. Math. Lett. \textbf{112 }(2021) 106775.
%%
%%
\bibitem{Kribs-Zaleta}
C. M.Kribs-Zaleta and J. X. Velasco-Hern\'andez,
\textit{A simple vaccination model with multiple endemic states}, Math. Biosci. \textbf{164} (2000) 183-201.
%%L
\bibitem{LaSalle}
J. P. Lasalle, 
\textit{The Stability of Dynamical Systems}, SIAM, Philadelphia, 1976.
\bibitem{Lyapunov}
A. M. Lyapunov, 
\textit{General Problem of the Stability Of Motion},  1st Edition, CRC Press, 1992.
%M
\bibitem{Martcheva}
M. Martcheva,
\textit{An Introduction to Mathematical Epidemiology}, Springer New York, NY, 2015.
%%
\bibitem{Martin-Vaquero1}
J. Mart\'in-Vaquero, A. Mart\'in del Rey, A. H. Encinas, J. D. Hern\'andez Guill\'en, A. Queiruga-Dios, and  G. Rodr\'iguez S\'anchez, 
\textit{Higher-order nonstandard finite difference schemes for a MSEIR model for a malware propagation}, J. Comput. Appl. Math.  \textbf{317} (2017) 146-156.
%%
\bibitem{Martin-Vaquero2}
J. Mart\'in-Vaquero, A. Queiruga-Dios, A.Mart\'in del Rey, A. H. Encinas, J. D. Hern\'andez Guill\'en, and G. Rodr\'iguez S\'anchez, 
\textit{Variable step length algorithms with high-order extrapolated non-standard finite difference schemes for a SEIR model}, J. Comput. Appl. Math. \textbf{330} (2018) 848-854.
%
%
\bibitem{Mattheij}
R. Mattheij and J. Molenaar, 
\textit{Ordinary Differential Equations in Theory and Practice}, Society for Industrial and Applied Mathematics, Philadelphia, 2002.
%%%
%%%
\bibitem{Mena-Lorca}
J. Mena-Lorca and H. W. Hethcote,
\textit{Dynamic models of infectious diseases as regulators of population sizes}, J. Math. Biol. \textbf{30} (1992) 693-716.
\bibitem{Mickens1}
R. E. Mickens, 
\textit{Nonstandard Finite Difference Models of Differential Equations}, World Scientific, 1993.
\bibitem{Mickens2}
R. E. Mickens, 
\textit{Applications of Nonstandard Finite Difference Schemes}, World Scientific, 2000.
\bibitem{Mickens3}
R. E. Mickens,\textit{ Advances in the Applications of Nonstandard Finite Difference Schemes}, World Scientific, 2005.
\bibitem{Mickens4}
R. E. Mickens, \textit{Nonstandard Finite Difference Schemes for Differential Equations}, J. Differ. Equ. Appl. \textbf{8} (2002) 823-847.
\bibitem{Mickens5}
R. E. Mickens, \textit{Nonstandard Finite Difference Schemes: Methodology and Applications},  World Scientific, 2020.
%%
%%
\bibitem{Mickens0}
R. E. Mickens,
\textit{Dynamic consistency: a fundamental principle for constructing nonstandard finite difference schemes for differential equations}, J. Differ. Equ. Appl. \textbf{11} (2005) 645-653.
%
\bibitem{Mickens6}
R. E. Mickens and T. M. Washington, \textit{NSFD discretizations of interacting population models satisfying conservation laws}, Comput. Math. with Appl. \textbf{66} (2013) 2307-2316.
\bibitem{Mickens7}
R. E. Mickens,
\textit{A nonstandard finite difference scheme for a PDE modeling combustion with nonlinear advection and diffusion}, Math. Comput. Simul. \textbf{69} (2005) 439-446.
%%
%%%
\bibitem{Mickens8}
R. E. Mickens,
\textit{A nonstandard finite difference scheme for the diffusionless Burgers equation with logistic reaction}, Math. Comput. Simul. \textbf{62} (2003) 117-124.
%%
%%
\bibitem{Mickens9}
R. E. Mickens, \textit{A nonstandard finite difference scheme for a Fisher PDE having nonlinear diffusion}, Comput. Math. with Appl.\textbf{ 45 }(2003) 429-436.
%%
\bibitem{Mickens10}
R. E. Mickens, \textit{A nonstandard finite difference scheme for a nonlinear PDE having diffusive shock wave solutions}, Math. Comput. Simul. \textbf{55} (2001) 549-555.
%%
\bibitem{Mickens11}
R. E. Mickens, \textit{A nonstandard finite-difference scheme for the Lotka-Volterra system}, Appl. Numer. Math. \textbf{45} (2003) 309-314
%%
\bibitem{Mickens12}
R. E. Mickens and I. H. Herron, 
\textit{Approximate rational solutions to the Thomas-Fermi equation based on dynamic consistency}, Appl. Math. Lett. \textbf{116} (2021) 106994.
%%
%%
\bibitem{Mickens13}
R. E. Mickens, K. Oyedeji, and S. Rucker, 
\textit{Exact finite difference scheme for second-order, linear ODEs having constant coefficients}, J. Sound Vib. \textbf{287} (2005) 1052-1056.
%%
%%
\bibitem{Mickens14}
R. E. Mickens and I. Ramadhani, 
\textit{Finite-difference schemes having the correct linear stability properties for all finite step-sizes III}, Comput. Math. with Appl.  \textbf{27} (1994) 77-84.
%%
%%
\bibitem{Mickens15}
R.E. Mickens and A. B. Gumel, 
\textit{Construction and analysis of a non-standard finite difference scheme for the Burgers-Fisher equation}, J. Sound Vib. \textbf{ 257} (2002) 791-797.
\bibitem{Mickens16}
R. E. Mickens, 
\textit{Numerical study of a non-standard finite difference scheme for the Val der Pol equation}, 
J. Sound Vib.  \textbf{250} (2002) 955-963.
%%
\bibitem{Mickens16a}
R. E. Mickens, \textit{NSFD scheme for acoustic propagation with the linearized Euler equations}, Math. Comput. Simul.  \textbf{127} (2016) 189-193.
%%%%%%%%%%%
%%%%%%%
\bibitem{Moghadas}
S. M. Moghadas and A. B. Gumel, 
\textit{Global stability of a two-stage epidemic model with generalized non-linear incidence}, Math. Comput. Simul. \textbf{60} (2002) 107-118.
%%
\bibitem{Patidar1}
K. C. Patidar, 
\textit{On the use of nonstandard finite difference methods}, J. Differ. Equ. Appl. \textbf{11} (2005) 735-758.
\bibitem{Patidar2}
K. C. Patidar, \textit{Nonstandard finite difference methods: recent trends and further developments}, J. Differ. Equ. Appl. \textbf{22} (2016) 817-849.
%%
\bibitem{Roeger1}
Lih-Ing W. Roeger,
\textit{Dynamically consistent discrete Lotka-Volterra competition models derived from nonstandard finite-difference schemes}, Discrete Contin. Dyn. Syst. - B. \textbf{9} (2008) 415-429.
%%
\bibitem{Roeger2}
Lih-Ing W. Roeger and G. L. Jr,
\textit{Dynamically consistent discrete Lotka-Volterra competition systems}, J. Differ. Equ. Appl. \textbf{19} (2013) 191-200
%
\bibitem{Roeger3}
Lih-Ing W. Roeger, 
\textit{Dynamically consistent discrete-time SI and SIS epidemic models}, Conference Publications, \textbf{2013}(special): 653-662. doi: 10.3934/proc.2013.2013.653.
\bibitem{Roeger4}
Lih-Ing W. Roeger,
\textit{Exact finite-difference schemes for two-dimensional linear systems with constant coefficients}, J. Comput. Appl. Math.  \textbf{219} (2008) 102-109.
%S
\bibitem{Smith}
H. L. Smith and P. Waltman, 
\textit{The Theory of the Chemostat: Dynamics of Microbial Competition}, Cambridge University Press, 2009.
%%
\bibitem{Stuart}
A. Stuart and A. R. Humphries, 
\textit{Dynamical Systems and Numerical Analysis}, Cambridge University Press, 1998.
%
\bibitem{Verma}
A. K. Verma, M. K. Rawani, and C. Cattani, 
\textit{A numerical scheme for a class of generalized Burgers' equation based on Haar wavelet nonstandard finite difference method},  Appl. Numer. Math. \textbf{168} (2021) 41-54.
%
\bibitem{Wood0}
D. T. Wood, D. T. Dimitrov, and H. V. Kojouharov, 
\textit{A nonstandard finite difference method for $n$-dimensional productive-destructive systems}, J. Differ. Equ. Appl.  \textbf{21} (2015) 240-254.
\bibitem{Wood1}
D. T. Wood and H. V. Kojouharov, \textit{A class of nonstandard numerical methods for autonomous dynamical systems}, 	Appl. Math. Lett. \textbf{50} (2015) 78-82.
%%%
\bibitem{Wood2}
D. T. Wood, H. V. Kojouharov, and D. T. Dimitrov,
\textit{Universal approaches to approximate biological systems with nonstandard finite difference methods}, Math. Comput. Simul. \textbf{133} (2017) 337-350.
\end{thebibliography}

\end{document}